\documentclass[10pt]{amsart}

\usepackage[colorlinks]{hyperref}
\usepackage{color,graphicx,shortvrb}
\usepackage[latin 1]{inputenc}
\usepackage[active]{srcltx} 

\usepackage{enumerate}
\usepackage{amssymb}

\newtheorem{theorem}{Theorem}[section]

\newtheorem{corollary}[theorem]{Corollary}

\newtheorem{lemma}[theorem]{Lemma}

\newtheorem{proposition}[theorem]{Proposition}
\newtheorem{remark}[theorem]{Remark}

\definecolor{barvicka}{rgb}{1,0,1}
\definecolor{Farbe}{rgb}{0,0,1}
\definecolor{barva}{rgb}{0,1,1}
\definecolor{colorido}{rgb}{1,0,0}

\def\J#1#2#3{ \left\{ #1,#2,#3 \right\} }
\def\RR{{\mathbb{R}}}

\def\11{\textbf{$1$}}

\newcommand{\norm}[1]{\left\|#1\right\|}
\newcommand{\ball}{{\mathcal B}}

\DeclareGraphicsExtensions{.jpg,.pdf,.png,.eps}

\begin{document}

\title{Preduals of JBW$^*$-triples are 1-Plichko spaces}

\author[M. Bohata, J. Hamhalter]{Martin Bohata, Jan Hamhalter}

\address{Czech Technical University in Prague, Faculty of Electrical Engineering, Department of Mathematics, Technicka 2, 166 27, Prague 6,
Czech Republic}
\email{bohata@math.feld.cvut.cz, hamhalte@math.feld.cvut.cz}

\author[O.F.K. Kalenda]{Ond\v{r}ej F.K. Kalenda}

\address{Charles University, Faculty of Mathematics and Physics, Department of
Mathematical Analysis, Sokolovsk{\'a} 86, 186 75 Praha 8, Czech Republic}
\email{kalenda@karlin.mff.cuni.cz}

\author[A.M. Peralta]{Antonio M. Peralta}

\address{Departamento de An{\'a}lisis Matem{\'a}tico, Facultad de
Ciencias, Universidad de Granada, 18071 Granada, Spain.}
\email{aperalta@ugr.es}

\author[H. Pfitzner]{Hermann Pfitzner}
\address{Universit\'{e} d'Orl\'{e}ans,
BP 6759,
F-45067 Orl\'{e}ans Cedex 2,
France}
\email{pfitzner@labomath.univ-orleans.fr}

\thanks{The first three authors were supported in part by the grant GA\v{C}R P201/12/0290. The fourth author was partially supported by the Spanish Ministry of Economy and Competitiveness and European Regional Development Fund project no. MTM2014-58984-P and Junta de Andaluc\'{\i}a grant FQM375.}



\date{}


\begin{abstract} We prove that the predual, $M_*$, of a JBW$^*$-triple $M$ is a 1-Plichko space
(i.e. it admits a countably 1-norming Markushevich basis or, equivalently,
it has a commutative 1-projectional skeleton), and obtain a natural description of the $\Sigma$-subspace of $M$.
This generalizes and improves similar  results for von Neumann algebras and JBW$^*$-algebras. Consequently, dual spaces of JB$^*$-triples also are 1-Plichko spaces.
We also show that $M_*$ is weakly Lindel\"{o}f determined if and only if $M$ is $\sigma$-finite if and only if $M_*$ is
weakly compactly generated.
Moreover, contrary to the proof for JBW$^*$-algebras, our proof dispenses with the use of elementary submodels theory.
\end{abstract}

\maketitle
\thispagestyle{empty}

\section{Introduction}
The topic of this paper concerns operator algebras, Jordan structures, and Banach\hyphenation{Banach} space theory. The main goal is to prove
that the predual of any JBW$^*$-triple satisfies the remarkable Banach space feature called 1-Plichko property.
The predual of a JBW$^*$-triple can be viewed as a non-commutative and non-associative  generalization of an $L^1$ space.
In  general such a space may be  highly non-separable.
Despite this fact, our result implies that the predual of a JBW$^*$-triple admits a nice decomposition into
separable subspaces and admits an appropriate Markushevich basis.
More precisely, let $X$ be a Banach space. A subspace $D\subset X^*$ is said to be a \emph{$\Sigma$-subspace} of $X^*$
if there is a linearly dense set $S\subset X$ such that $$ D =\{ \phi \in X^* : \{m \in S : \phi(m) \neq 0\} \hbox{ is countable} \}.$$
The Banach space $X$ is called \emph{($r$-)Plichko} if $X^*$ admits a ($r$-)norming $\Sigma$-subspace,
i.e. there exists a \emph{$\Sigma$-subspace} $D$ of $X^*$ such that $$\|x\| \leq r \sup \{ |\phi (x) | : \phi\in D, \|\phi\|\leq 1\}\ \ (x\in X)$$
(compare \cite{Kal2000,Kal2008}). The 1-Plichko property is equivalent to the fact that $X$ has a
countably 1-norming Markushevich basis \cite[Lemma 4.19]{Kal2000}. Another deep result \cite[Theorem 27]{Kubis} says
that $X$ is a 1-Plichko space if and only if it admits a commutative 1-projectional skeleton. A commutative 1-projectional
skeleton is a system $(P_\lambda)_{\lambda\in \Lambda }$ of norm one projections on $X$, where $\Lambda$ is an up-directed set,
fulfilling the following conditions:
\begin{itemize}
\item $P_\lambda X$ is separable for each $\lambda$ and $X=\bigcup_{\lambda\in \Lambda}P_\lambda X$.
\item $P_\lambda P_\mu= P_\lambda$ whenever $\lambda\le \mu$.
\item $P_\lambda P_\mu= P_\mu P_\lambda$ for all $\lambda$ and $\mu$.
\item if $(\lambda_n)$ is an increasing net in $\Lambda$, it has a supremum, $\lambda\in \Lambda$, and\\
 $P_\lambda X= \overline{\bigcup_nP_{\lambda_n}X}$.
\end{itemize}

It easily follows that any 1-Plichko space enjoys the 1-separable complementation property saying that any separable subspace can be enlarged to a 1-complemented separable subspace. This property was established by U. Haagerup for preduals of von Neumann algebras with the help of results from modular theory of von Neumann algebras (see \cite[Theorem~IX.1]{GhoGoMauScha87}).\smallskip

The category of 1-Plichko spaces involves many classes of spaces studied in Banach space theory. Let us recall that $X$ is \emph{weakly Lindel\"of determined}, WLD in short, if $X^*$ is a $\Sigma$-subspace of itself. $X$ is called \emph{weakly compactly generated} (WCG in short) if it contains a weakly compact subset whose linear span is dense in $X$. Obviously, every WLD space is 1-Plichko, and it follows from \cite[Proposition 2]{AmirLind68} that every WCG space is WLD. Plichko and 1-Plichko spaces were formally introduced in \cite[\S 4.2]{Kal2000}. The notion was motivated by a series of papers where A.N. Plichko studied this property under equivalent reformulations (see \cite{Plichko79, Plichko82, Plichko83, Plichko86}). Although the term 1-Plichko is the most commonly used name for the spaces defined above, they have been also studied under different names. Namely, the class of those Banach spaces which are 1-Plichko is precisely the class termed $\mathcal{V}$ by J. Orihuela in \cite{Ori}, which has been also studied by M. Valdivia in \cite{Val91}.\smallskip

It has been proved by the third author of this note in \cite{Kal2008} that many important spaces have 1-Plichko property, for example $L^1$ spaces for non-negative $\sigma$-finite measures, order-continuous Banach lattices, and $C(K)$-spaces for abelian compact groups $K$.
Moreover, the paper \cite{Kal2008} contains  the first result on non-commutative $L^1$ spaces  showing that the predual of
a semi-finite von Neumann algebra is 1-Plichko. Motivated by the latter, the first three authors of this paper prove in \cite{BoHamKalIsrael} that the predual of any von Neumann algebra is 1-Plichko.
Moreover, they showed that the canonical 1-norming $\Sigma$-subspace  is the space of all elements whose range projection is $\sigma$-finite. A generalization to JBW$^*$-algebras appeared to be non-trivial. In \cite{BoHamKal2016} the same authors showed that the predual of any JBW$^*$-algebra is 1-Plichko. The proof was quite different from that given in the setting of von Neumann algebras. The proof in the Jordan case was based on constructing a special projection skeleton with the help of the set theoretical tool of elementary submodels. Obviously, the question whether, as in the case of von Neumann algebra preduals \cite{BoHamKalIsrael}, the result can be obtained without any use of submodels theory is a gap which is not fulfilled by the results in \cite{BoHamKal2016}.\smallskip

In the present paper we prove a further generalization of the above mentioned results by showing that all JBW$^\ast$-triple preduals are 1-Plichko spaces. Our main result reads as follows.

\begin{theorem}\label{t:main} The predual $M_*$ of a JBW$^*$-triple $M$ is a 1-Plichko space. Moreover,
$M_*$ is weakly Lindel\"of determined if and only if $M$ is $\sigma$-finite. In this case $M_*$ is even weakly compactly generated.
\end{theorem}

The approach in this paper resembles more the one of \cite{BoHamKalIsrael} than the one of \cite{BoHamKal2016}. One reason for this has already been mentioned, in the present paper the proofs and arguments do not make use of the set theoretic tool of submodels. Moreover, the theory of JBW$^*$-triples allows to connect the description of the $\Sigma$-subspace obtained in \cite{BoHamKalIsrael} and to obtain a similar and satisfactory
description for JBW$^*$-triples (and hence also for JBW$^*$-algebras), see Theorem \ref{T:unique}. The key result for this approach is Proposition \ref{p sigma-finiteness of bicontractively complemented subtriples}.\smallskip

The relevant notions related to JBW$^*$-triples are gathered in Section~\ref{sec:prelim}.
Theorem \ref{t:main} -- in fact a more precise version of Theorem \ref{t:main} -- follows from
Theorems~\ref{T:sigmafinite} and~\ref{t predual of JBW*-triples are 1-Plichko without submodels} proved below.\smallskip

Since the second dual of a JB$^*$-triple is a JBW$^*$-triple (see \cite[Corollary\ 3.3.5]{Chu2012}), the next result is a
straightforward consequence of Theorem~\ref{t:main}.

\begin{corollary}\label{c dual of JB*-triples are 1-Plichko} The  dual space of a JB$^*$-triple is a 1-Plichko space.$\hfill\Box$
\end{corollary}

We recall that a Banach space $X$ has the (\emph{$r$-{\rm)}separable complementation property} if any separable subspace of $X$ is contained in a ($r$-)complemented separable subspace of $X$ (compare \cite[page 92]{GhoGoMauScha87}).
Since $1$-Plichko spaces enjoy the $1$-separable complementation property (which follows immediately from the characterization
using a projectional skeleton formulated above), we also get the following result.

\begin{corollary}\label{c preduals of JBW*-triples have 1SCP} Preduals of JBW$^*$-triples have the 1-separable complementation property. $\hfill\Box$
\end{corollary}

The above corollary is an extension of a result of U. Haagerup, who showed that the same statement holds for von Neumann algebra preduals (with different methods, see \cite[Theorem IX.1]{GhoGoMauScha87}).\bigskip

\section{Notation and preliminaries}\label{sec:prelim}

In this section we recall basic notions and results on JBW$^*$-triples and Plichko spaces. We also include
some auxilliary results needed to prove our main results.
For unexplained notation from Banach space theory we refer to \cite{FHHMZ}. The symbols $\ball_X$ and $X^*$ will denote the closed unit ball and the dual of a Banach space $X$, respectively.

\subsection{\bf{Elements of JBW$^*$-triples}}\label{triples}
In \cite{Ka83}, W. Kaup obtains an analytic-algebraic characterization of bounded symmetric domains in terms of the so-called
JB$^*$-triples, by showing that every bounded symmetric domain in a complex Banach space is biholomorphically equivalent to the
open unit ball of a JB$^*$-triple. Thanks to this result, JB$^*$-triples offer a natural bridge to connect the infinite-dimensional
holomorphy with functional analysis. We recall that a JB$^*$-triple is a complex Banach space $E$ equipped with a continuous
ternary product $\{.,.,.\}$, which is symmetric and bilinear in the outer variables and conjugate-linear in the middle one, satisfying
the following properties:
\begin{itemize}
	\item $\J xy{\J abc} = \J {\J xya}bc - \J a{\J yxb}c + \J ab{\J xyc}$ for all $a,b,c,x,y\in E$ (Jordan identity),
	\item the operator $x\mapsto\J aax$ is a hermitian operator with nonnegative spectrum for each $a\in E$,
	\item $\|\J aaa\|=\|a\|^3$ for $a\in E$.
\end{itemize}
We recall that an operator $T\in B(E)$ is hermitian if and only if $\|\exp(irT)\|=1$ for each $r\in\RR$.
For $a,b\in E$ we define a (linear) operator $L(a,b)$ on $E$ by $L(a,b)(x)=\J abx$, $x\in E$,
and a conjugate-linear operator $Q(a,b)$ by $Q(a,b)(x)=\J axb$. Given $a\in E$,
the symbol $Q(a)$ will denote the operator on $E$ defined by $Q(a)=Q(a,a)$.\smallskip

Every C$^*$-algebra is a JB$^*$-triple with respect to the triple product given by \hyphenation{product} $\J xyz =\frac12 (x y^* z +z y^* x).$ The same triple product equips the space $B(H,K)$, of all bounded linear operators between
complex Hilbert spaces $H$ and $K$, with a structure of a JB$^*$-triple. Among the examples involving Jordan algebras,
we can say that every JB$^*$-algebra is a JB$^*$-triple under the triple product $\J xyz = (x\circ y^*) \circ z + (z\circ y^*)\circ x - (x\circ z)\circ y^*$.\smallskip

An element $e$ in a JB$^*$-triple $E$ is said to be a \emph{tripotent} if $e= \J eee$.
If $E$ is a von Neumann algebra viewed as a JBW$^*$-triple, then any projection is clearly a tripotent; in fact, an element of a von Neumann algebra is a tripotent
if and only if it is a partial isometry.\smallskip

For each tripotent $e\in E$, the mappings $P_{i} (e) : E\to E$ $(i=0,1,2)$ defined by
$$P_2 (e) = L(e,e)(2 L(e,e) -id_{E}), \ P_1 (e) = 4
L(e,e)(id_{E}-L(e,e))$$ $$ \ \hbox{ and } P_0 (e) =
(id_{E}-L(e,e)) (id_{E}-2 L(e,e))$$ are contractive linear projections (see \cite[Corollary 1.2]{FriRu85}), called the \emph{Peirce} \emph{projections} associated with $e$.
It is known (cf.\ \cite[p.\ 32]{Chu2012}) that $P_2(e)=Q(e)^2$, $P_1(e)=2\left(L(e,e)-Q(e)^2\right)$,
and $P_0(e)=id_{E}-2L(e,e)+Q(e)^2$.
In case $E$ is a von Neumann algebra, $e\in E$ a partial isometry, $q=e^*e$ the initial projection and $p=ee^*$ the final projection, we get
$$P_2(e)x=pxq,\ P_1(e)x=px(1-q)+(1-p)xq \mbox{ and } P_0(e)x=(1-p)x(1-q).$$
If $e$ is even a symmetric element (i.e. $e^*=e$) in the von Neumann algebra then we have $p=q$.\smallskip

The range of $P_i(e)$ is the eigenspace, $E_i(e)$, of $L(e, e)$ corresponding to the eigenvalue $\frac{i}{2},$ and
$$E= E_{2} (e) \oplus E_{1} (e)\oplus E_0 (e)$$ is termed the \emph{Peirce decomposition} of $E$ relative to $e$.
Clearly, $e\in E_2(e)$ and $P_k(e)(e)=0$ for $k=0,1$.
The following multiplication rules (known as Peirce rules or Peirce arithmetic) are satisfied:
\begin{equation}
\label{eq peirce rules1} \J {E_{2} (e)}{E_{0}(e)}{E} = \J {E_{0}
(e)}{E_{2}(e)}{E} =\{0\},
\end{equation} \begin{equation}
\label{eq peirce rules2} \J {E_{i}(e)}{E_{j} (e)}{E_{k}
(e)}\subseteq E_{i-j+k} (e),
\end{equation} where $E_{i-j+k} (e)=\{0\}$ whenever
$i-j+k \notin \{ 0,1,2\}$ (\cite{FriRu85} or \cite[Theorem\ 1.2.44]{Chu2012}).
A tripotent $e$ is called \emph{complete} if $E_0 (e) =\{0\}$.
The complete tripotents of a JB$^*$-triple $E$ are precisely the complex and the real
extreme points of its closed unit ball (cf. \cite[Lemma 4.1]{BraKaUp78} and \cite[Proposition 3.5]{KaUp77} or
\cite[Theorem 3.2.3]{Chu2012}).
Therefore every JBW$^*$-triple contains an abundant collection of complete tripotents.
If $E= E_2(e)$, or equivalently, if $\{e,e,x\}={x}$ for all $x\in E$, we say that $e$ is \emph{unitary}.\smallskip

For each tripotent $e$ in a JB$^*$-triple, $E$, the Peirce-2 subspace $E_2 (e)$ is a unital JB$^*$-algebra with unit $e$,
product $a\circ_{e} b := \{ a,e,b\}$ and involution $a^{*_e} := \J eae$ (cf. \cite[\S 1.2 and Remark 3.2.2]{Chu2012}).
As we noticed above, every JB$^*$-algebra is a JB$^*$-triple with respect to the product
$$\{a,b,c\}=(a \circ b^*) \circ c +(c \circ b^*) \circ a - (a \circ c) \circ b^*.$$
Kaup's Banach-Stone theorem (see \cite[Proposition 5.5]{Ka83}) assures that a surjective operator between JB$^*$-triples is
an isometry if and only if it is a triple isomorphism.
Consequently, the triple product on $E_2 (e)$ is uniquely determined by the expression
\begin{equation}\label{eq product Peirce2 as JB*-algebra} \{ a,b,c\} =(a \circ_{e} b^{*_e}) \circ_{e} c +(c \circ_{e} b^{*_e}) \circ_e a - (a \circ_e c) \circ b^{*_e},
\end{equation}
for every $a,b,c\in E_2 (e)$. Therefore, unital JB$^*$-algebras are in one-to-one correspondence with JB$^*$-triples
admitting a unitary element (see also \cite[4.1.55]{CabRodPal2014}).\smallskip

A JBW$^*$-triple is a JB$^*$-triple which is also a dual Banach space. Examples of JBW$^*$-triples include von Neumann algebras
and JBW$^*$-algebras.  Every JBW$^*$-triple admits a unique isometric predual and its triple product is separately weak$^*$-to-weak$^*$-continuous (\cite{BarTi}, \cite{Horn87}, \cite[Theorem\ 3.3.9]{Chu2012}).
Consequently, the Peirce projections associated with a tripotent in a JBW$^*$-triple are weak$^*$-to-weak$^*$-continuous.
Therefore, for each tripotent $e$ in a JBW$^*$-triple $M$, the Peirce subspace $M_2 (e)$ is a JBW$^*$-algebra. Unlike  general JB$^*$-triples, JBW$^*$-triples admit
a rather concrete representation which we recall in Section~\ref{sec:structure} below as it is the essential tool for proving our
results.\smallskip

Let $a,b$ be elements in a JB$^*$-triple $E$. Following standard terminology, we shall say that $a$ and $b$ are
\emph{algebraically orthogonal} or simply \emph{orthogonal} (written $a\perp b$) if $L(a,b) =0$. If we consider a C$^*$-algebra $A$
as a JB$^*$-triple, then two elements $a,b\in A$ are orthogonal in the C$^*$-sense (i.e. $a b^* = b^* a=0$) if and only if they
are orthogonal in the triple sense.  Orthogonality is a symmetric relation.
By Peirce arithmetic it is immediate that all elements in $E_2(e)$ are orthogonal to all elements in $E_0(e)$,
in particular, two tripotents $u,v\in E$ are orthogonal if and only if $u\in E_0(v)$ (and, by symmetry, if and only if $v\in E_0(u)$).
We refer to \cite[Lemma~1]{BurFerGarMarPe} for other useful characterizations of orthogonality and additional details not explained here.\smallskip

The order in the partially ordered set  of all tripotents in a JB$^*$-triple $E$ is defined as follows:
Given two tripotents $e,u\in E$, we say that $e\leq u$ if $u-e$ is a tripotent which is orthogonal to $e$.

\begin{lemma}{\rm(}\cite[Cor.\ 1.7]{FriRu85}, \cite[Prop.\ 1.2.43]{Chu2012}{\rm)} Let $u,e$ be two tripotents in a JB$^*$-triple $E$. The following assertions are equivalent.
\begin{enumerate}[$(1)$]
\item $e\leq u$.
\item $P_2(e)(u)= e$.
\item $\{u,e,u\}=e$.
\item $\{e,u,e\}=e$.
\item $e$ is a projection {\rm(}i.e. a self-adjoint idempotent{\rm)} in the JB$^*$-algebra $E_2(u)$.
\end{enumerate}
\end{lemma}

For each norm-one functional $\varphi$ in the predual, $M_*$, of a JBW$^*$-triple $M$, there exists a
unique tripotent $e\in M$ satisfying $\varphi = \varphi P_{2} (e)$ and $\varphi|_{M_{2}(e)}$ is a faithful normal
state of the JBW$^*$-algebra $M_{2} (e)$ (see \cite[Proposition 2]{FriRu85}). This unique tripotent $e$ is called the \emph{support tripotent} of $\varphi$, and will be denoted by $e(\varphi)$. It is explicitly shown in \cite[part $(b)$ in the proof of Proposition 2]{FriRu85} that \begin{equation}\label{eq order in support tripotents}\hbox{if $u$ is a tripotent in $M$ with $1=\|\varphi\| = \varphi (u)$, then $u \geq e(\varphi)$.}
\end{equation}\smallskip

We recall that a subspace $I$ of a JB$^*$-triple $E$ is called an \emph{inner ideal}, provided $\J IEI\subseteq I$ (i.e., provided $\J abc\in I$ whenever $a,c\in I$ and $b\in E$). Clearly, an inner ideal is a subtriple.
Note that if $e$ is a tripotent of a JBW$^*$-triple $M$, then $M_2(e)$ is a weak$^*$-closed subtriple of $M$
(\cite[Th.\ 1.2.47]{Chu2012}).
In a von Neumann algebra $W$ (regarded as JBW$^*$-triple)
left and right ideals and sets of the form $aWb$ (with fixed $a,b\in W$) are inner ideals, whereas weak$^*$-closed inner ideals are of the form
$pWq$ with projections $p,q\in W$ \cite[Thm.\ 3.16]{EdRu88bis}.\smallskip

Given an element $x$ in a JB$^*$-triple $E$ the symbol $E_x$ will denote the norm-closed subtriple of $E$ generated by $x$,
that is, the closed subspace generated by all odd powers $x^{[2n+1]}$, where $x^{[1]}=x,$ $x^{[3]}=\{x,x,x\}$, and $x^{[2n+1]}=\{x,x,x^{[2n-1]}\}$ ($n\geq 2$) (compare, for example, \cite[Sec. 3.3]{Loos} or \cite[Lemma 1.2.10]{Chu2012}).
It is known that there exists an isometric triple isomorphism $\Psi : E_x\to C_{0} (L)$ satisfying $\Psi (x) (t) = t,$ for all $t$ in
$L$ (compare \cite[1.15]{Ka83}), where $C_{0} (L)$ is the abelian C$^*$-algebra of all complex-valued
continuous functions on $L$ vanishing at $0$, $L$ being a locally compact subset of $(0,\|x\|]$ satisfying that $L\cup\{0\}$ is
compact.
Thus, for any continuous function $f: L\cup\{0\} \to \mathbb{C}$ vanishing at 0, it is possible to give the usual meaning
in the sense of functional calculus to $f(x)\in E_x,$ via $f(x)=\Psi^{-1}(f)$.\smallskip

For each norm-one element $x$ in a JBW$^*$-triple $M$, $r(x)$ will denote its  \emph{range tripotent}.
We succinctly describe its definition. (More details are given for example in \cite[Section  2.2]{PePfiJMAA2016}
or in \cite[comments before Lemma 3.1]{EdRu88} or \cite[\S 2]{BuChuZa}). For $x\in M$ with $\norm{x}=1$, the functions $t\to t^{\frac{1}{2n-1}}$
give rise to an increasing sequence $(x^{[\frac{1}{2n-1}]})$ which weak$^*$-converges to $r(x)$ in $M$.
The tripotent $r(x)$ is the smallest tripotent $e\in M$ satisfying that $x$ is a positive element in the JBW$^*$-algebra $M_{2} (e)$ (see, for example, \cite[comments before Lemma 3.1]{EdRu88} or \cite[\S 2]{BuChuZa}).
The inequality $x\leq r(x)$ holds in $M_2(r(x))$ for every norm-one element $x\in E$. For a non-zero element $z\in M$, the range tripotent of $z$, $r(z)$, is precisely the range tripotent of $\frac{z}{\|z\|}$, and we set $r(0)=0$.\smallskip

Let $M$ be a JBW$^*$-triple. We recall that a tripotent $u$ in $M$ is said to be \emph{$\sigma$-finite} if $u$ does not majorize
an uncountable orthogonal subset of tripotents in $M$. Equivalently, $u$ is a $\sigma$-finite tripotent in $M$ if and only if
there exists an element $\varphi$ in $M_*$ whose support tripotent $e(\varphi)$ coincides with $u$ (cf. \cite[Theorem 3.2]{EdRu98}).
Following standard notation, we shall say that $M$ is $\sigma$-finite if every tripotent in $M$ is $\sigma$-finite, equivalently,
every orthogonal subset of tripotents in $M$ is countable (cf. \cite[Proposition 3.1]{EdRu98}).
It is also known that the sum of an orthogonal countable family of mutually orthogonal $\sigma$-finite tripotents in $M$ is again
a $\sigma$-finite tripotent (see \cite[Theorem 3.4$(i)$]{EdRu98}). It is further proved in \cite[Theorem 3.4$(ii)$]{EdRu98}
that every tripotent in $M$ is the supremum of a set of mutually orthogonal $\sigma$-finite tripotents in $M$.\smallskip

When a von Neumann algebra $W$ is regarded as a JBW$^*$-triple, a projection $p$ is $\sigma$-finite in the triple sense
if and only if it is $\sigma$-finite or countably decomposable in the usual sense employed for von Neumann algebras
(compare \cite[Definition 2.1.8]{Sak} or \cite[Definition II.3.18]{Tak}).\smallskip

We will need the following properties of $\sigma$-finite tripotents which have been borrowed from \cite{EdRu98}.

\begin{lemma}\label{L:sigmafinitetrip}\cite{EdRu98}
Let $M$ be a JBW$^*$-triple and let $e$ be a tripotent of $M$. Then the following hold:
\begin{enumerate}[$(i)$]
	\item $M_2(e)$ is a JBW$^*$-subtriple of $M$ and any tripotent $p\in M_2(e)$ is $\sigma$-finite in $M_2(e)$ if and only if it is $\sigma$-finite in $M$.
	\item $e$ is $\sigma$-finite if and only if $M_2(e)$ is $\sigma$-finite.
	\item If $e$ is $\sigma$-finite, then any tripotent in $M_2(e)$ is $\sigma$-finite in $M$.
\end{enumerate}
\end{lemma}

\begin{proof} Since $M_2(e)$ is a weak$^*$-closed subtriple of $M$, assertion $(i)$ follows from
\cite[Lemma 3.6(ii)]{EdRu98}.
Assertion $(ii)$ follows from $(i)$, \cite[Theorem 4.4 (viii)-(ix)]{EdRu98} and the fact that $e$ is a complete tripotent in $M_2(e)$.
Finally, assertion $(iii)$ follows immediately from $(i)$ and $(ii)$.
\end{proof}

{For non explained notions concerning JB$^*$-algebras and  JB$^*$-triples we refer to the monographs
\cite{CabRodPal2014} and \cite{Chu2012}.}

\subsection{Contractive and bicontractive projections}\label{subsec: contractive projections}

One of the main properties enjoyed by any member $E$ in the class of JB$^*$-triples states that the image of a
contractive projection $P: E \rightarrow E$  (where contractive means $\|P\|\le 1$) is again a JB$^*$-triple with triple product
$\{x,y,z\}_{P} := P( \J xyz )$ for $x,y,z$ in $P(E)$ and
\begin{equation}\label{eq bicontractive proj}
P \J axb = P \J a{P(x)}b,\qquad a,b\in P(E), x\in E,
\end{equation}  (see \cite{Ka2}, \cite{Sta} and \cite{FriRu4}). It is further known that under these conditions $P(E)$ need not be,
in general, a JB$^*$-subtriple of $E$ (compare \cite[Example 1]{FriRu82} or \cite[Example 3]{Ka2}).
But note that if $P(E)$ is known to be a subtriple then $\{\cdot,\cdot,\cdot\}_{P}$ coincides with the original triple product
of $E$ because in JB$^*$-triples norm and triple product determine each other (see e.g. \cite[Th. 3.1.7, 3.1.20]{Chu2012}).
Fortunately, more can be said about the JB$^*$-triple structure of $P(E)$. It is known that $P(E)$ is isometrically isomorphic
to a JB$^*$-subtriple of $E^{**}$ (see \cite[Theorem 2]{FriRu87}).\smallskip

If $P:E\to E$ is even a bicontractive projection (where bicontractive means $\|P\|\le 1$ and $\|I-P\|\le1$ --
by $I_V$ or simply $I$ we denote the identity on a vector space $V$) on a JB$^*$-triple,
it satisfies a stronger property. Namely, $P(E)$ is then a JB$^*$-subtriple of $E$, in particular \eqref{eq bicontractive proj}
can be improved because the identities
\begin{equation}\label{eq bicontractive proj 2}
P\{a,b,x\}=\{a,b,P(x)\}\quad\mbox{ and }\quad P\{a,x,b\}=\{a,P(x),b\}
\end{equation} hold for $a,b\in P(E)$, $x\in E$ (cf.\ \cite[\S 3]{FriRu87}). It is further known that when $P$ is bicontractive,
there exists a surjective linear isometry (i.e. a triple automorphism) $\Theta$ on $E$ of period 2 such that $P=\frac12 (I+\Theta)$
(see \cite[Theorem 4]{FriRu87}). Since, by another interesting property of JBW$^*$-triples, every surjective linear isometry
on a JBW$^*$-triple is weak$^*$-to-weak$^*$-continuous (see \cite[Proof of Theorem 3.2]{Horn87}) we have, as a consequence, that a bicontractive projection $P$ on a JBW$^*$-triple is weak$^*$-to-weak$^*$-continuous.

\subsection{Von Neumann tensor products}\label{sec:tensor}

We recall now some basic facts on von Neumann tensor products of von Neumann algebras. The theory has been essentially borrowed from \cite[Chapter IV]{Tak}, and we refer to the latter monograph for additional results not commented here.
Let $A\subset {B}(H)$ and $W\subset {B}(K)$ be von Neumann algebras. The algebraic tensor product $A\otimes W$ is canonically embedded into ${B}(H\otimes K)$, where $H\otimes K$ is the hilbertian tensor product of $H$ and $K$ (see \cite[Definition IV.1.2]{Tak}). The von Neumann algebra generated by the algebraic tensor product $A\otimes W$ is denoted $A\overline{\otimes} W,$ and is called \emph{the von Neumann tensor product} of $A$ and $W$. Note that $A\overline{\otimes} W$ is the weak$^*$ closure
of $A\otimes W$ in ${B}(H\otimes K)$ (see \cite[\S IV.5]{Tak}). \smallskip

If $A$ is commutative, then the predual of $A\overline{\otimes} W$ is canonically identified with the projective tensor product of preduals, i.e.
\begin{equation}
\label{eq:predualtensor}(A\overline{\otimes} W)_*=A_*\widehat{\otimes}_{\pi}W_*.
\end{equation}

This follows from \cite[Theorem IV.7.17]{Tak} (or rather \cite[Section IV.7]{Tak}). Furthermore, the special case of a separable $W_*$ is treated in \cite[Th.\ 1.22.13]{Sak}, while there is another approach via results on operator spaces: Results due to E.G. Effros and Z.J. Ruan show that equality \eqref{eq:predualtensor} holds for any von Neumann algebra $W$, when the projective tensor product
on the right-hand side is in the category of operator spaces (\cite{EfRu1990}, \cite[Theorem 7.2.4]{EfRu}).
But if $A$ is commutative, it carries the minimal operator-space structure by \cite[Proposition 3.3.1]{EfRu} and hence
the predual $A_*$ carries the maximal structure by
\cite[(3.3.13) or (3.3.15) on p. 51]{EfRu}, and hence by
\cite[(8.2.4) on p. 146]{EfRu} the projective tensor product in the category of operator spaces coincides with the
projective tensor product in the Banach space sense.

\begin{lemma}\label{L:contractiveproj}
Let $A$ and $W$ be von Neumann algebras with $A$ commutative. Suppose $P: W\to W$ is a weak$^*$-to-weak$^*$-continuous contractive projection.
Then the following holds:
\begin{enumerate}[$(i)$]
	\item $P(W)$ is a JBW$^*$-triple with triple product $\{x,y,z\}_{P} := P( \J xyz )$ for $x,y,z$ in $P(W)$.
	\item $A\overline{\otimes} P(W)$, the weak$^*$-closure of the algebraic tensor product $A\otimes P(W)$ in  $A \overline{\otimes} W$, is the range of a weak$^*$-to-weak$^*$-continuous contractive projection $Q$ on $A\overline{\otimes} W$.
	\item $A\overline{\otimes} P(W)$ is a JBW$^*$-triple  with triple product $\{x,y,z\}_{Q} := Q( \J xyz )$ for $x,y,z$ in $A\overline{\otimes}P(W)$.
	            Moreover, $$(A\overline{\otimes} P(W))_*=A_*\widehat{\otimes}_{\pi}(P(W))_*=A_*\widehat{\otimes}_{\pi}P^*(W_*).$$
\end{enumerate}
\end{lemma}

\begin{proof} We know from Section~\ref{subsec: contractive projections} that statement $(i)$ is satisfied.\smallskip

Since $P$ is weak$^*$-to-weak$^*$ continuous, it is the dual map of a map $P_*:W_*\to W_*$.
It is clear that $P_*$ is a contractive projection on $W_*$.
It follows from basic tensor product properties (cf. \cite[3.2]{DefFlo} or \cite[Proposition 2.3]{Ryan2002}) that $I\otimes P_*$
is a contractive projection on $A_*\widehat{\otimes}_{\pi}W_*$.
Moreover, by \cite[3.8]{DefFlo} or \cite[Proposition 2.5]{Ryan2002} the norm on its range (which is the norm-closure of the algebraic tensor product $A_*\otimes P_*(W_*)$) is the projective norm coming from $A_*\widehat{\otimes}_{\pi}P_*(W_*)$.\smallskip

Further, it is clear that the dual mapping $Q=(I\otimes P_*)^*$ is a weak$^*$-to-weak$^*$-continuous contractive projection on $(A_*\widehat{\otimes}_{\pi}W_*)^*=A\overline{\otimes}W$. Using the results commented in Section~\ref{subsec: contractive projections} we know that its range is a JBW$^*$-triple with the triple product defined in $(iii)$.
Since the range of $Q$ is canonically identified with the dual of $A_*\widehat{\otimes}_{\pi}P_*(W_*)$, to complete the proof of $(ii)$ and $(iii)$ it is enough to show that the range of $(I\otimes P_*)^*$ is $A\overline{\otimes} P(W)$.\smallskip

To show the desired equality we observe that the restriction of $(I\otimes P_*)^*$ to the algebraic tensor product $A\otimes W$ coincides with
$I\otimes P$. Therefore the range of $(I\otimes P_*)^*$ contains $A\otimes P(W)$ and hence also its weak$^*$ closure
$A\overline{\otimes} P(W)$.
Conversely, since the unit ball $\ball_{A\otimes W}$ is weak$^*$-dense in $\ball_{A\overline{\otimes} W}$ (for example by the
Kaplansky density theorem), and $(I\otimes P_*)^*$ is weak$^*$-to-weak$^*$-continuous, $\ball_{A\otimes  W}$ is weak$^*$ dense
in the unit ball of the range of $(I\otimes P_*)^*$ as well. This completes the proof.
\end{proof}

\begin{lemma}\label{l lifting bicontractive projections to vN tensor products} Let $A$ and $W$ be von Neumann algebras with $A$ commutative. Suppose $P: W\to W$ is a bicontractive projection. Then the following holds:
\begin{enumerate}[$(i)$]
	\item $P(W)$ is a JBW$^*$-subtriple of $W$.
	\item $A\overline{\otimes} P(W)$, the weak$^*$-closure of the algebraic tensor product $A\otimes P(W)$ in  $A \overline{\otimes} W$, is the range of a bicontractive projection on $A\overline{\otimes} W$.
	\item $A\overline{\otimes} P(W)$ is a JBW$^*$-subtriple of $A\overline{\otimes} W$ and, moreover,
	$$(A\overline{\otimes} P(W))_*=A_*\widehat{\otimes}_{\pi}(P(W))_*=A_*\widehat{\otimes}_{\pi}P^*(W_*).$$
\end{enumerate}
\end{lemma}

\begin{proof} By Section~\ref{subsec: contractive projections} we know that $P(W)$ is a JB$^*$-subtriple of $W$ and that $P$ is weak$^*$-to-weak$^*$-continuous. Hence we can apply Lemma~\ref{L:contractiveproj}. Moreover, since $P$ is even bicontractive, we get that $P_*$ is bicontractive, and hence $I\otimes P_*$ and $Q=(I\otimes P_*)^*$ are bicontractive too. Finally, since $Q$ is bicontractive, by Section~\ref{subsec: contractive projections} we get that $A\overline{\otimes} P(W)$ is a JBW$^*$-subtriple of $A\overline{\otimes} W$.
\end{proof}

\subsection{Structure theory}\label{sec:structure}

In this subsection we recall an important structure result, due to G. Horn \cite{Ho2} and G. Horn and E. Neher \cite{HoNe},
which allows us to represent every JBW$^*$-triple in a concrete way. These results will be the main tool for proving
that JBW$^*$-triple preduals are 1-Plichko spaces.\smallskip

We begin by recalling the definition of Cartan factors.\label{def Cartan factors}
There are six types of them (compare \cite[Example 2.5.31]{Chu2012}):\smallskip

Type 1: A Cartan factor of type 1 coincides with a Banach space $B(H, K),$ of all bounded linear operators between two complex Hilbert spaces $H$ and $K$, where the triple product is defined by $\J xyz= 2^{-1}(xy^*z+zy^*x)$. If $\dim H=\dim K$, then we
can suppose $H=K$ and we get the von-Neumann algebra $B(H)$. If $\dim K<\dim H$, we may suppose that $K$ is a closed subspace of $H$ and then $B(H,K)$ is a JB$^*$-subtriple of $B(H)$. Moreover, if $p$ is the orthogonal projection of $H$ onto $K$, then
$x\mapsto px$ is a bicontractive projection of $B(H)$ onto $B(H,K)$. If $\dim K>\dim H$, we may suppose that $H$ is a closed subspace of $K$, $p$ the orthogonal projection of $K$ onto $H$ and then $x\mapsto xp$ is a bicontractive projection
of $B(K)$ onto $B(H,K)$.\smallskip

Types 2 and 3:
Cartan factors of types 2 and 3 are the subtriples of $B(H)$ defined by $C_2 = \{ x\in B(H) : x=- j x^* j\} $ and
$C_3 = \{ x\in B(H) : x= j x^* j\}$, respectively,
where $j$ is a conjugation (i.e. a conjugate-linear isometry of period $2$) on $H$.
If $j$ is a conjugation on $H$, then there is an orthonormal basis $(e_\gamma)_{\gamma\in \Gamma}$ such that
$j(\sum_{\gamma\in\Gamma} c_\gamma e_\gamma)=\sum_{\gamma\in\Gamma}\overline{c_\gamma} e_\gamma$.
Each $x\in B(H)$ can be represented by a ``\emph{matrix}'' $(x_{\gamma\delta})_{\gamma,\delta\in\Gamma}$.
It is easy to check that the representing matrix of $jx^*j$ is the transpose of the representing matrix of $x$.
Hence, $C_2$ consists of operators with antisymmetric representing matrix and $C_3$ of operators with symmetric ones.
Therefore, $P(x)=\frac12(x^t+x)$ (where $x^t=jx^*j$ is the transpose of $x$ with respect to the basis chosen above)
is a bicontractive projection on $B(H)$ such that $C_3$ is the range of $P$, and $C_2$ is the range of $I-P$.\smallskip

Type 4: A Cartan factor of type 4 (denoted by $C_4$) is a complex spin factor, that is, a complex Hilbert space (with inner product $\langle .,. \rangle$) provided with a conjugation $x \mapsto \overline{x}$, triple product $$\J x y z
= \langle x , y \rangle z + \langle z , y \rangle x - \langle x ,
\bar z \rangle \bar y,$$ and norm given by $\| x\|^2=\langle x , x
\rangle+\sqrt {\langle x , x \rangle^2-|\langle x , \overline x
\rangle|^2}$. We point out that $C_4$ is isomorphic to a Hilbert space and hence, in particular, reflexive.\smallskip

Types 5 and 6: All we need to know about Cartan factors of types 5 and 6 (also called \emph{exceptional Cartan} factors) is that they are all finite dimensional.\smallskip

Although H. Hanche-Olsen showed in \cite[\S 5]{Hanche-Olsen1983} that the standard method to define tensor products of JC-algebras
(and JW$^*$-triples) is, in general, hopeless, von Neumann tensor products can be applied in the representation theory of
JBW$^*$-triples. Let $A$ be a commutative von Neumann algebra and let $C$ be a Cartan factor which can be realised as a
JW$^*$-subtriple of some $B(H)$.
As before,
the symbol $A \overline{\otimes} C$ will denote the
weak$^*$-closure {of the algebraic tensor product $A\otimes C$} in the usual von Neumann tensor product
$A \overline{\otimes} B(H)$ of $A$ and $B(H)$.
This applies to Cartan factors of types 1--4 (this is obvious for Cartan factors of types 1--3, the case of type 4 Cartan factors follows from \cite[Theorem 6.2.3]{HanStor}).
\smallskip

The above construction does not cover Cartan factors of types 5 and 6. When $C$ is an exceptional Cartan factor, $A \overline{\otimes} C$ will denote the injective tensor product of $A$ and $C$, which can be identified with the space $C(\Omega, C)$, of all continuous functions on $\Omega$ with values in $C$ endowed with the pointwise operations and the supremum norm, where $\Omega$ denotes the spectrum of $A$ (cf. \cite[p. 49]{Ryan2002}).
 We observe that if $C$ is a finite dimensional Cartan factor which can be realised as a JW$^*$-subtriple of some $B(H)$ both
 definitions above give the same object (cf. \cite[Theorem IV.4.14]{Tak}).\smallskip

The structure theory settled by G. Horn and E. Neher \cite[(1.7)]{Ho2}, \cite[(1.20)]{HoNe} proves that every JBW$^*$-triple $M$ writes (uniquely up to triple isomorphisms) in the form
\begin{equation}\label{eq decomposition of JBW*-triples}
M = \left(\bigoplus_{j\in \mathcal{J}} A_j \overline{\otimes} C_j \right)_{\ell_{\infty}}
\oplus_{\ell_{\infty}} H(W,\alpha)\oplus_{\ell_{\infty}} pV,
\end{equation}
where each $A_j$ is a commutative von Neumann algebra, each $C_j$ is a Cartan factor, $W$ and $V$ are continuous von
Neumann algebras, $p$ is a projection in $V$, $\alpha$ is {a linear} involution on $W$ commuting with $^*$, that is,
a linear $^*$-antiautomorphism of period 2 on $W$, and $H(W,\alpha)=\{x\in W: \alpha(x)=x\}$.

\subsection{Some facts on Plichko spaces}

The following lemma sums up several basic properties of $\Sigma$-subspaces.

\begin{lemma}\label{L:sigmasubspace}
Let $X$ be a Banach space and $S\subset X^*$ a $\Sigma$-subspace. Then the following hold:
\begin{enumerate}[$(i)$]
	\item $S$ is weak$^*$-countably closed. That is, $\overline{C}^{w^*}\subset S$ whenever $C\subset S$ is countable.
	In particular, $S$ is weak$^*$-sequentially closed and norm-closed.
	\item Bounded subsets of $S$ are weak$^*$-Fr\'echet Urysohn. That is, given $A\subset S$ bounded and $x^*\in S$ such that $x^*\in\overline{A}^{w^*}$, then there is a sequence $(x_n^*)$ in $A$ weak$^*$-converging to $x^*$.
	\item Let $S'\subset X^*$ be any other subspace satisfying (i) and (ii). If $S\cap S'$ is 1-norming, then $S=S'$.
	\item If $X$ is WLD, then $X^*$ is the only norming $\Sigma$-subspace of $X^*$.
	\item If $S$ is $1$-norming, then for any $x\in X$ there is $x^*\in S$ of norm one such that $x^*(x)=\|x\|$.
\end{enumerate}
\end{lemma}

\begin{proof} Assertion $(i)$ follows from the very definition of a $\Sigma$-subspace, assertion $(ii)$ follows from \cite[Lemma 1.6]{Kal2000}.
Assertion $(iii)$ is an easy consequence of $(i)$ and $(ii)$ and follows from \cite[Lemma 2]{Kal2002} (in fact in the just quoted lemma
it is assumed that $S'$ is a $\Sigma$-subspace as well, but the proof uses only properties $(i)$ and $(ii)$).
Assertion $(iv)$ follows immediately from $(iii)$ and $(v)$ is an easy consequence of $(i)$.
\end{proof}

We will also need the following easy lemma on quotients of $1$-Plichko spaces.

\begin{lemma}\label{L:quotient} Let $X$ be a $1$-Plichko Banach space, and let $S\subset X^*$ be a $1$-norming $\Sigma$-subspace.
Suppose that $Z\subset X^*$ is a weak$^*$-closed subspace such that $S\cap \ball_Z$ is weak$^*$ dense in $\ball_Z$. Then
$S\cap Z$ is a $1$-norming $\Sigma$-subspace of $Z=(X/Z_\perp)^*$.
\end{lemma}

\begin{proof} Since $Z$ is a weak$^*$-closed subspace of the dual space $X^*$, it is canonically isometrically identified with
$(X/Z_\perp)^*$. Further, by the assumptions it is clear that $S\cap Z$ is a $1$-norming subspace of $Z$. It remains to show it is a $\Sigma$-subspace. \smallskip

To do that, fix a linearly dense set $A\subset X$ such that
$$S=\{x^*\in X^*: \{x\in A: x^*(x)\ne 0\}\mbox{ is countable}\}.$$
Let $\tilde A$ be the image of $A$ in $X/Z_\perp$ by the canonical quotient mapping. It is clear that
$\tilde A$ is linearly dense. Let
$$\tilde S = \{ x^*\in Z =(X/Z_\perp)^*: \{x\in \tilde A: x^*(x)\ne 0\}\mbox{ is countable}\}$$
be the $\Sigma$-subspace induced by $\tilde A$. It is easy to check that $S\cap Z\subset \tilde S$.
It follows from Lemma~\ref{L:sigmasubspace}(iii) that $S\cap Z=\tilde S$, which completes the proof.
\end{proof}

\section{Preduals of $\sigma$-finite JBW$^*$-triples}\label{sec:3}


The aim of this section is to prove the following result.

\begin{theorem}\label{T:sigmafinite} The predual of any $\sigma$-finite JBW$^*$-triple is weakly compactly generated,
in fact even Hilbert generated.
\end{theorem}

Recall that a Banach space $X$ is said to be {\it Hilbert-generated} if there is a Hilbert space $H$ and a bounded linear
mapping $T:H\to X$ with dense range. It is clear that any Hilbert-generated Banach space is weakly compactly generated
(the generating weakly compact set is precisely $T(\ball_H)$).\smallskip

Theorem \ref{T:sigmafinite} above follows from the following stronger statement,
which is a JBW$^*$-triple analogue of \cite[Lemma 3.3]{BoHamKalIsrael} for von Neumann algebras and of \cite[Proposition 3.7]{BoHamKal2016} for JBW$^*$-algebras.

\begin{proposition}\label{prop Peirce 2 and 1 predual are WCG} Let $e$ be a $\sigma$-finite tripotent in a JBW$^*$-triple $M$.
Then the predual of the space $M_2 (e) \oplus M_1 (e)$ {\rm(}i.e. $(P_2(e)+P_1(e))^* (M_*)${\rm)} is Hilbert-generated.
\end{proposition}

To see that Theorem~\ref{T:sigmafinite} follows from the above proposition it is enough to use the fact that any JBW$^*$-triple
contains an abundant set of complete tripotents.
In particular, any $\sigma$-finite JBW$^*$-triple $M$ contains a $\sigma$-finite complete tripotent $e\in M$ such that
$M=M_2(e)\oplus M_1(e)$. Hence Proposition \ref{prop Peirce 2 and 1 predual are WCG} entails Theorem~\ref{T:sigmafinite}.\smallskip

Next let us focus on the proof of Proposition~\ref{prop Peirce 2 and 1 predual are WCG}. Similarly as in the case of
von Neumann algebras and JBW$^*$-algebras it will be done by introducing a canonical (semi)definite inner product.
In \cite[Proposition 1.2]{BarFri1987}, Barton and Friedman showed that given an element $\varphi$ in the dual of a JB$^*$-triple $E$ and an element $z\in E$ such that $\varphi(z)=\|\varphi\|=\|z\|=1$, the map $E\times E\ni(x,y)\mapsto \langle x, y\rangle_{\varphi}:=\varphi\{x,y,z\}$ defines a hermitian semi-positive sesquilinear form with the associated pre-hilbertian seminorm $\|x\|_{\varphi}:=(\varphi\{x,x,z\})^{1/2}$ on $M$ and is independent of $z$.\smallskip

We shall need the following  technical lemma borrowed from \cite[Lemma 4.1]{EdRu98}:

\begin{lemma}\label{l hilbertian norm}
Let $M$ be a JBW$^*$-triple, let $\varphi\in M_*$ be of norm one and let $e=e(\varphi)\in M$ be its support tripotent. Then the annihilator of the pre-hilbertian seminorm  $\|\cdot\|_{\varphi}$  is precisely $M_0(e)$, that is,
\begin{equation}\label{eqn kernel of seminorm} \{ x\in M:
\norm{x}_{\varphi}=0\}= M_0(e).
\end{equation}
In particular, the restriction of $\norm{\cdot}_{\varphi}$ to $M_2(e)\oplus M_1(e)$ is a pre-hilbertian norm and
the restriction of $\langle\cdot,\cdot\rangle_{\varphi}$ to $M_2(e)\oplus M_1(e)$ is an inner product.
\end{lemma}

\begin{proof} The first statement is proved in \cite[Lemma 4.1]{EdRu98}, the positive definiteness of $\norm{\cdot}_{\varphi}$ and of $\langle\cdot,\cdot\rangle_{\varphi}$ on $M_2(e)\oplus M_1(e)$ follows immediately (see also \cite[Lemma 1.5]{FriRu85}, \cite{Pe15}).
\end{proof}

Now we are ready to prove the main proposition of this section:

\begin{proof}[Proof of Proposition~\ref{prop Peirce 2 and 1 predual are WCG}]
Since $e$ is a $\sigma$-finite tripotent there exists a norm-one normal functional $\varphi\in M_*$ such that $e = e(\varphi)$
is the support tripotent of $\varphi$. Denote by $h_{\varphi}$ the pre-hilbertian space $M_2(e)\oplus M_1(e)$ equipped with
the inner product $\langle \cdot, \cdot\rangle_{\varphi}=\varphi\{\cdot,\cdot,e\}$, and write $H_{\varphi}$ for its completion.
Let us first consider $\widetilde\Phi(a)$ defined by $x\mapsto\langle x,a\rangle_{\varphi}$ for $a\in h_{\varphi}$, $x\in M$.
By the Cauchy-Schwarz inequality we have
$$|\widetilde\Phi(a)(x)|=|\langle x,a\rangle_{\varphi}|\leq\norm{x}_{\varphi}\norm{a}_{\varphi}\leq\norm{x}\norm{a}_{\varphi}$$
which, together with the separate $w^*$-continuity of the triple product, shows that $\widetilde\Phi$ is a well-defined conjugate-linear
contractive map from $h_{\varphi}$ to $M_*$.\smallskip

In order to see that the range of $\widetilde\Phi$ is contained in $(M_2(e)\oplus M_1(e))_*=(P_2^*(e)+P_1^*(e))(M_*)$,
let us observe that for any $a\in h_\varphi$ and $y\in M_0(e)$, we have $\norm{y}_{\varphi}=0$ by Lemma \ref{l hilbertian norm}, and
hence $\widetilde\Phi(a)(y)=0$.

Thus, by density of $h_\varphi$ in $H_\varphi$, $\widetilde\Phi=(P_2^*(e)+P_1^*(e))\widetilde\Phi$ gives rise to a
conjugate-linear continuous map $\Phi:H_{\varphi}\to (M_2(e)\oplus M_1(e))_*$.\smallskip

We shall finally prove that $\Phi $ has norm-dense range.
Suppose $z\in M_2(e)\oplus M_1 (e)$ satisfies $\Phi (a)  (z) =0$ for every $a\in h_{\varphi}$.
In particular, $0= \Phi (z) (z) = \norm{z}_{\varphi}^2$ and thus, by Lemma \ref{l hilbertian norm}, $z=0$.
By the Hahn-Banach theorem, $\Phi$ has dense range. If we replace the map $\Phi$ by $\Phi j$, where $j$ is a conjugation on $H_{\varphi}$, then we have a linear mapping.
\end{proof}

\section{The case of general JBW$^*$-triples}\label{sec:4}

{In this section we state and prove Theorem \ref{t predual of JBW*-triples are 1-Plichko without submodels}, which gives a more
precise version of the first part of Theorem~\ref{t:main}.\smallskip

To provide a precise formulation we introduce one more notation.
For a JBW$^*$ triple $M$ we define the set
\begin{eqnarray*}
M_\sigma =\{x\in M: \mbox{ there is a $\sigma$-finite tripotent $e\in M$ such that }P_2(e)x=x\}
\end{eqnarray*}
and note that
$$\begin{aligned}
M_\sigma&=\{x\in M: \mbox{ there is a $\sigma$-finite tripotent $e\in M$ such that }\J eex=x\}
\\&= \{x\in M: \mbox{ $r(x)$ is a $\sigma$-finite tripotent  }\}
.\end{aligned}
$$
Indeed, the first equality follows from the fact that the range of $P_2(e)$ is the eigenspace of $L(e,e)$ corresponding
to the eigenvalue $1$. Let us show the second equality.
The inclusion `$\supset$' is obvious. To show the converse inclusion, let $x\in M_\sigma$.
Fix a $\sigma$-finite tripotent $e\in M$ with $x=P_2(e)x$, i.e., $x\in M_2(e)$. Since $M_2(e)$ is a JBW$^*$-subtriple of $M$
and $r(x)$ belongs to the JBW$^*$-subtriple generated by $x$, we have $r(x)\in M_2(e)$ and so $r(x)$ is $\sigma$-finite
by Lemma~\ref{L:sigmafinitetrip}.\medskip

We mention the easy but useful fact that $M_{\sigma}$ is 1-norming in $M$.
To see this we simply observe that $M_{\sigma}$ contains all $\sigma$-finite tripotents of $M$, or equivalently,
all support tripotents of functionals in $M_*$.\smallskip

\begin{theorem}\label{t predual of JBW*-triples are 1-Plichko without submodels}
The predual space of a JBW$^*$-triple $M$ is a 1-Plichko space. Moreover,
\begin{equation}
\label{eq:Msigma}
M_{\sigma}\mbox{ is a $1$-norming $\Sigma$-subspace of } M=(M_*)^*.
\end{equation}
In particular, $M_*$ is weakly Lindel\"of determined if and only if $M$ is $\sigma$-finite.
\end{theorem}
}

It is not obvious that $M_\sigma$ is a subspace, but this will follow by the proof of
Theorem \ref{t predual of JBW*-triples are 1-Plichko without submodels}; it will be reproved a second time
in Theorem \ref{T:unique}.\medskip

The `in particular' part of the theorem is an immediate consequence of the first statements of the theorem.
Indeed,  $M$ is $\sigma$-finite if and only if $M=M_\sigma$ (cf. Lemma \ref{L:sigmafinitetrip}).
Hence, if $M$ is $\sigma$-finite, then $M_*$ is WLD by the first statement. Conversely, if $M_*$ is WLD,
then by the first part of {the} theorem together with Lemma~\ref{L:sigmasubspace}$(iv)$ we get  $M=M_\sigma$, hence $M$ is $\sigma$-finite.
Thus, it is enough to prove \eqref{eq:Msigma}. This will be done in the rest of this section
by using results in \cite{BoHamKalIsrael} and the decomposition \eqref{eq decomposition of JBW*-triples}.\smallskip

{The following proposition is almost immediate from the main results of \cite{BoHamKalIsrael}.}

\begin{proposition}\label{P:mainvN}
The statement of Theorem~\ref{t predual of JBW*-triples are 1-Plichko without submodels} holds for von Neumann algebras.
\end{proposition}

\begin{proof}
{It is enough to show \eqref{eq:Msigma} in case $M$ is a von Neumann algebra.}
In view of \cite[Proposition 4.1]{BoHamKalIsrael}, {to this end} it is enough to show that
$$M_\sigma=\{x\in M: x=qxq\mbox{ for a $\sigma$-finite projection }q\in M\}.$$

Let $x$ be in the set on the right-hand side. Fix a $\sigma$-finite projection $q\in M$ with $x=qxq$. Then $q$ is a $\sigma$-finite tripotent and $\J qqx=\frac12(qx+xq)=qxq=x$. Hence $x\in M_\sigma$. \smallskip

Conversely, let $x\in M_\sigma$ and let $u\in M$ be  a $\sigma$-finite triponent with $x=P_2(u)x$. Since $M$ is a von Neumann algebra, $u$ is a partial isometry and hence $P_2(u)x=pxq$, where $p=uu^*$ is the final projection and $q=u^*u$ is the initial projection. Then $p$ is a $\sigma$-finite projection.
Indeed, suppose that $(r_\gamma)_{\gamma\in \Gamma}$ is an uncountable family of pairwise orthogonal projections smaller than $p$. Then it is easy to check that $(r_\gamma u)_{\gamma\in \Gamma}$ is an uncountable family of pairwise orthogonal tripotents smaller than $u$. Similarly we get that $q$ is $\sigma$-finite.
Hence their supremum $r=p\vee q$ is $\sigma$-finite as well {(\cite[Theorem 3.4]{EdRu98} or \cite[Exercice 5.7.45]{KadRing})}
and satisfies $x=rxr$. Thus $x$ belongs to the set on the right-hand side and the proof is complete.
\end{proof}

\begin{proposition}\label{p sigma-finiteness of bicontractively complemented subtriples} Let $P : M\to M$ be a bicontractive
projection on a JBW$^*$-triple, let $N= P(M)$, and let $e$ be a tripotent in $N$.
Then $e$ is $\sigma$-finite in $N$ if and only if $e$ is $\sigma$-finite in $M$, that is, $N_{\sigma} = N\cap M_{\sigma}$.
\end{proposition}

\begin{proof}
The ``if'' implication is clear. Let $e$ be a $\sigma$-finite tripotent in $N$. By \cite[Theorem 3.2]{EdRu98} there exists a norm-one functional $\phi\in N_*$ whose support tripotent in $N$ is $e$. Let us define $\psi = P^* (\phi) = \phi P \in M_*$. Clearly $\|\psi\|=1$. We shall prove that $e$ is the support tripotent of $\psi$ in $M$, and hence $e$ is $\sigma$-finite in $M$ (\cite[Theorem 3.2]{EdRu98}). Let $u$ be the support tripotent of $\psi$ in $M$. From $\psi (e) = \phi (e) = 1= \|\psi\|$ we get $e\geq u$ (compare \cite[part $(b)$ in the proof of Proposition 2]{FriRu85}).\smallskip

We set $u_1=P(u)$ and $u_2=u-u_1$. Since $e\geq u$ in $M$, we deduce that $\{ e,u, e\}=u=\{e,e,u\}$ ($e-u\in M_0(u)$ and Peirce rules). Hence, $u_1=P(u)=\{ e,Pu, e\}=\{e,u_1,e\}$ and $u_1=\{e,e,u_1\}$ by \eqref{eq bicontractive proj 2}.
It follows that $u_1 = \{e,u_1,e\}\in M_2 (e)$ and that $u_1= \{e,u_1,e\}= u_1^{*_e}$ is a hermitian element in the closed unit ball of the JBW$^*$-algebra $N_2 (e)$.
As $e$ is the unit in this algebra and $u_1$ is a hermitian element of norm less than one, we see that $e-u_1$ is a positive element in the JBW$^*$-algebra $N_2(e)$.
The condition $$\phi(e)= 1= \psi (u) = \phi P (u) = \phi (u_1)$$ implies, by the faithfulness of $\phi|_{N_2 (e)}$, that $u_1 = e$.\smallskip

It follows from the above that $u_2=\{e,e,u\}-\{e,e,u_1\}=\{e,e,u_2\}$ and similarly $u_2=\{e,u_2,e\}$.
These identities combined with the fact that $u = e+ u_2$ is a tripotent (that is, $\{e+ u_2,e+ u_2,e+ u_2\} = e+ u_2$) yield
$$ e+ u_2 = e   + 2 \{u_2,u_2,e\}+ \{u_2, e, u_2\}+  3 u_2 + \{u_2,u_2,u_2\}.$$
After applying the bicontractive projection $I-P$ in both terms of the last equality we get $-2 u_2 = \{u_2,u_2,u_2\}$.
Now $2\|u_2\|=\|\{u_2,u_2,u_2\}\|=\|u_2\|^3$ implies either $u_2=0$ or $\|u_2\|^2=2$. The latter is not possible because
$\|u_2\|\leq1$ by the fact that $u_2=(I-P)u$ and $I-P$ is a contraction. Thus $u_2=0$, and hence $e= u$, which proves the first statement.\smallskip

For the last identity we observe that for every element $x\in N$, its range tripotent $r(x)$ (in $N$ or in $M$) lies in $N$. Suppose $x$ is an element in $N$ whose range tripotent is $\sigma$-finite in $N$. We deduce from the first statement that $r(x)$  is also $\sigma$-finite in $M$, and hence $N_{\sigma}\subseteq M_{\sigma}$.
The inclusion $N_{\sigma}\supseteq M_{\sigma}\cap N$ is clear.
\end{proof}

By combining Proposition~\ref{P:mainvN}, Proposition~\ref{p sigma-finiteness of bicontractively complemented subtriples}, and
Lemma~\ref{L:quotient} we get the following proposition.

\begin{proposition}\label{p 1-Plichko for predulas of images of normal bicontractive projections} Let $P : W\to W$ be a
bicontractive projection on a von Neumann algebra $W$, let $M= P(W)$. Then $M_*$ is a 1-Plichko space. Furthermore, $M_{\sigma}$
is a $1$-norming $\Sigma$-subspace of $M$.
\end{proposition}

Now we are ready to prove the validity of \eqref{eq:Msigma} for most of the summands from the representation \eqref{eq decomposition of JBW*-triples}:

\begin{proposition}\label{P:summandsbicontractive}
Let $M$ be a JBW$^*$-triple of one of the following forms:
\begin{enumerate}[$(a)$]
	\item $M=A\overline{\otimes}C$, where $A$ is a commutative von Neumann algebra and $C$ is a Cartan factor of type 1, 2 or 3.
	\item $M=H(W,\alpha)$, where $W$ is a von Neumann algebra and $\alpha$ is a linear involution on $W$ commuting with $^*$.
	\item $M=pV$, where $V$ is a von Neumann algebra and $p\in V$ is a projection.
\end{enumerate}
Then $M_\sigma$ is a $1$-norming $\Sigma$-subspace of $M=(M_*)^*$.
\end{proposition}

\begin{proof} We will apply Proposition~\ref{p 1-Plichko for predulas of images of normal bicontractive projections}.
To do that it is enough to show that $M$ is the range of a bicontractive projection on a von Neumann algebra.\smallskip

$(a)$ If $C$ is a Cartan factor of type 1, 2 or 3, then $C$ is the range of a bicontractive projection on a certain von Neumann algebra $W,$ as
it was previously observed after the definitions of the respective Cartan factors. The desired bicontractive projection on $A\overline{\otimes}W$ is finally given by Lemma~\ref{l lifting bicontractive projections to vN tensor products}. \smallskip

$(b)$ A bicontractive projection on $W$ is given by $x\mapsto \frac12(x+\alpha(x))$. \smallskip

$(c)$ The mapping $x\mapsto px$ defines a bicontractive projection on $V$.
\end{proof}

The remaining summands from \eqref{eq decomposition of JBW*-triples} are covered by the following theorem, which we
formulate in a more abstract setting of Banach spaces.

\begin{theorem}\label{t Bochner 1-Plichko} Let $(\Omega,\Sigma,\mu)$ be a measure space with a non-negative semifinite measure, and let $E$ be a reflexive Banach space. Then the space $L^1(\mu,E)$ of Bochner-integrable functions is $1$-Plichko. Furthermore, $L^1(\mu,E)$ is weakly Lindel\"{o}f determined if and only if $\mu$ is $\sigma$-finite, in the latter case it is even weakly compactly generated.

More precisely, there is a family of finite measures $(\mu_\gamma)_{\gamma\in\Gamma}$ such that $L^1(\mu,E)$
is isometric to
$$\left(\bigoplus_{\gamma\in\Gamma} L^1(\mu_\gamma,E)\right)_{\ell_1}$$
and
$$S=\left\{ f=(f_\gamma)_{\gamma\in\Gamma}\in \left(\bigoplus_{\gamma\in\Gamma} L^\infty(\mu_\gamma,E)\right)_{\ell_\infty}:
\{\gamma\in\Gamma : f_\gamma\ne 0\}\mbox{ is countable}\right\}$$
is a $1$-norming $\Sigma$-subspace of  $(L^1(\mu,E))^*=\left(\bigoplus_{\gamma\in\Gamma} L^\infty(\mu_\gamma,E)\right)_{\ell_\infty}$.
\end{theorem}

\begin{proposition}\label{p Bochner WCG} Let $\mu$ be a finite measure, and let $E$ be a reflexive Banach space. Then $L^1(\mu,E)$ is weakly compactly generated.
\end{proposition}

\begin{proof} The proof is done similarly as in the scalar case (cf. \cite[Theorem 5.1]{Kal2008}). Let us consider the identity mapping $T:L^2(\mu,E)\to L^1(\mu,E)$.
By the Cauchy-Schwarz inequality we get $\|T\|\le\sqrt{\|\mu\|}$, hence $T$ is a bounded linear operator. Moreover, the range of $T$ is dense, since countably valued functions in $L^1(\mu,E)$ are dense in the latter space. Finally, $L^2(\mu,E)$ is reflexive because $E$ and $E^*$ have Radon-Nikod\'{y}m property (see
\cite[Theorem IV.1.1]{DieUhl}). Thus, $L^1(\mu,E)$ is indeed weakly compactly generated.
\end{proof}

Remark: Note that if $E$ is isomorphic to a Hilbert space, then we can even conclude that $L^1(\mu,E)$ is Hilbert generated, since in this
case $L^2(\mu,E)$ is also isomorphic to a Hilbert space. Indeed, if $E$ is even isometric to a Hilbert space, the norm on
$L^2(\mu,E)$ is induced by the scalar product $$\langle f,g\rangle = \int \langle f(\omega),g(\omega)\rangle d\mu(\omega).$$

\begin{proof}[Proof of Theorem \ref{t Bochner 1-Plichko}]
We imitate the proof of \cite[Theorem 5.1]{Kal2008}. Let $\mathcal{B}\subset\Sigma$ be a maximal family  with the following
properties:
\begin{itemize}
	\item $0<\mu(B)<+\infty$ for each $B\in \mathcal{B}$;
	\item $\mu(B_1\cap B_2)=0$ for each $B_1,B_2\in\mathcal{B}$ distinct.
\end{itemize}
The existence of such a family follows immediately from Zorn's lemma.\smallskip

Take any separable-valued $\Sigma$-measurable function $f:\Omega\to E$.  Then clearly
$$\int \|f(\omega)\| d\mu(\omega)=\sum_{B\in\mathcal{B}}\int_B \|f(\omega)\| d\mu(\omega).$$
Therefore $L^1(\mu,E)$ is isometric to the $\ell_1$-sum of spaces $L^1(\mu_{|B},E)$, $B\in\mathcal{B}$. Since $\mu_{|B}$ is finite for each $B\in\mathcal{B}$, $L^1(\mu_{|B},E)$ is weakly compactly generated (and hence weakly Lindel\"{o}f determined) by the previous Proposition \ref{p Bochner WCG}. Further, it is clear that the dual of $L^1(\mu,E)$ is canonically isometric to the $\ell_\infty$-sum of the family $\{ (L^1(\mu_{|B},E))^* : B\in\mathcal{B}\}$. More concretely, since $E$ is reflexive, by \cite[Theorem IV.1.1]{DieUhl} we have $(L^1(\mu_{|B},E))^*= L^\infty(\mu_{|B},E^*)$ for each $B\in\mathcal{B},$ and hence
$$L^1(\mu,E)^*= \left(\bigoplus_{B\in\mathcal B} L^\infty(\mu_{|B},E^*)\right)_{\ell_\infty}.$$
Finally, it follows from \cite[Lemma 4.34]{Kal2000} that
$$S=\left\{ (f_B)_{B\in\mathcal B}\in \left(\bigoplus_{B\in\mathcal B} L^\infty(\mu_{|B},E^*)\right)_{\ell_\infty}:
\{B\in\mathcal B; f_B\ne 0\}\mbox{ is countable}\right\}$$
is a $1$-norming $\Sigma$-subspace of $(L^1(\mu,E))^*$.\smallskip

To prove the last statement, it is enough to observe that $\mu$ is $\sigma$-finite if and only if $\mathcal{B}$ is countable,
that a countable $\ell_1$-sum of weakly compactly generated spaces is again weakly compactly generated and that an
uncountable $\ell_1$-sum of nontrivial spaces contains $\ell_1(\omega_1)$ and hence is not weakly Lindel\"{o}f determined.
(Recall that WLD property passes to subspaces.)
\end{proof}

\begin{proposition}\label{P:Cartanpredual}
Let $A$ be a  commutative von Neumann algebra and $C$ a Cartan factor. Then $(A\overline{\otimes}C)_*= A_* \widehat{\otimes}_{\pi} C_*$.
\end{proposition}

\begin{proof}
If $C$ is a Cartan factor of type 1, 2 or 3, then $C$ is the range of a bicontractive projection on a von Neumann
algebra and hence the equality follows from Lemma~\ref{l lifting bicontractive projections to vN tensor products}.\smallskip

If $C$ is a type 4 Cartan factor, it follows from \cite[Lemma 2.3]{EffStor} that $C$ is the range of a
(unital positive) contractive projection
$P:B(H)\to B(H)$ where $H$ is an appropriate Hilbert space.
{The mapping $P^{**}: B(H)^{**}\to B(H)^{**}$ is a weak$^*$-to-weak$^*$-continuous contractive projection
on the von Neumann algebra $B(H)^{**}$ whose range is $C$ by (Goldstine's theorem and) reflexivity of $C$.}
Hence the desired equality follows from Lemma~\ref{L:contractiveproj}.\smallskip

If $C$ is a Cartan factor of type 5 or 6, then it is finite-dimensional and $A\overline{\otimes}C$ is defined to be the
injective tensor product. Further, by \cite[3.2]{DefFlo} or \cite[p. 24]{Ryan2002} we get $(A_* \widehat{\otimes}_{\pi} C_*)^*=B(A_*,C)$ which coincides
with the injective tensor product $A\widehat{\otimes}_{\varepsilon} C$, as $C$ has finite dimension.
\end{proof}

\begin{lemma}\label{L:ellinftysum} Let $(M_\gamma)_{\gamma\in\Gamma}$ be an indexed family of JBW$^*$-triples, and let us denote $\displaystyle M=\left(\bigoplus_{\gamma\in\Gamma} M_\gamma\right)_{\ell_\infty}.$ Then
$$M_\sigma = \Big\{ (x_\gamma)_{\gamma\in\Gamma}\in M: x_\gamma\in (M_{\gamma})_\sigma\mbox{ for }\gamma\in\Gamma
\ \&\ \{\gamma\in\Gamma : x_\gamma\ne 0\}\mbox{ is countable}\Big\}.$$
\end{lemma}

\begin{proof} This follows easily if we observe that $e=(e_\gamma)_{\gamma\in\Gamma}\in M$ is a tripotent if and only if $e_\gamma$ is a tripotent for each $\gamma$ and, moreover, $e$ is $\sigma$-finite if and only if each $e_\gamma$ is $\sigma$-finite and only countably many $e_\gamma$ are nonzero.
\end{proof}

\begin{proposition}\label{P:summandsBochner}
Let $A$ be a commutative von Neumann algebra and $C$ a reflexive Cartan factor. {\rm(}This applies, in particular, to Cartan factors of types 4, 5 and 6.{\rm)} Let $M=A\overline{\otimes}C$. Then $M_\sigma$ is a $1$-norming $\Sigma$-subspace of
$M=(M_*)^*$, and hence $M_*$ is $1$-Plichko. Furthemore,  $M_*$ is  weakly Lindel\"{o}f determined if and only if $A$ is $\sigma$-finite. In such a case $M_*$ is even weakly compactly generated.
\end{proposition}

\begin{proof}
If $A$ is a commutative von Neumann algebra, by \cite[Theorem III.1.18]{Tak} it can be represented as $L^\infty(\Omega,\mu)$,
where $\Omega$ is a locally compact space and $\mu$ a positive Radon measure on $\Omega$. In fact, $\Omega$ is the topological sum of a family of compact spaces $(K_\gamma)_{\gamma\in\Gamma}$. Then the predual of $A$ is identified with
$$L^1(\Omega,\mu)=\left(\bigoplus_{\gamma\in\Gamma} L^1(K_\gamma,\mu|_{K_\gamma})\right)_{\ell_1}.$$
Since
$$(A\overline{\otimes}C)_*=A_* \widehat{\otimes}_{\pi} C_*=L^1(\mu,C_*),$$
we can use Theorem~\ref{t Bochner 1-Plichko}. To complete the proof it is enough to show that $S=M_\sigma$, where
$S$ is the $\Sigma$-subspace provided by Theorem~\ref{t Bochner 1-Plichko}.
Since
$$M=\left(\bigoplus_{\gamma\in\Gamma} L^\infty(K_\gamma,\mu|_{K_\gamma},C)\right)_{\ell_\infty},$$
due to Lemma~\ref{L:ellinftysum}, it is enough to show that $L^\infty(\mu,C)$ is $\sigma$-finite whenever
$\mu$ is finite. But, in this case, its predual, $L^1(\mu,C_*),$ is weakly compactly generated by Proposition~\ref{p Bochner WCG}, thus
$L^\infty(\mu,C)$ is $\sigma$-finite by Theorem~\ref{t Bochner 1-Plichko}.
\end{proof}

\begin{proof}[Proof of Theorem \ref{t predual of JBW*-triples are 1-Plichko without submodels}]
{We have already mentioned that it is enough to show  \eqref{eq:Msigma}.}
Let $M$ be a JBW$^*$-triple and consider the decomposition \eqref{eq decomposition of JBW*-triples}.
By Propositions~\ref{P:summandsbicontractive} and~\ref{P:summandsBochner} each summand fulfills \eqref{eq:Msigma}.
Further, Lemma~\ref{L:ellinftysum} {and} \cite[Lemma 4.34]{Kal2000} yield the validity of \eqref{eq:Msigma} for $M$.
\end{proof}

In passing we remark that from Theorem \ref{t predual of JBW*-triples are 1-Plichko without submodels} (and the general facts on Plichko spaces)
we have that $M_\sigma$ is norm-closed and even weak$^*$-countably closed; it is additionally weak$^*$-closed if and only if $M$ is $\sigma$-finite.

\section{Structure of the space $M_\sigma$}

In the previous section we proved that, for any JBW$^*$-triple $M$, $M_\sigma$ is a $1$-norming $\Sigma$-subspace of $M=(M_*)^*$.
If $M$ is $\sigma$-finite, it is the only $1$-norming $\Sigma$-subspace and coincides with the whole $M$. If $M$ is not $\sigma$-finite,
there may be plenty of different $1$-norming $\Sigma$-subspaces (cf. \cite[Example 6.9]{Kal2000}). However, $M_\sigma$ is the only
canonical $1$-norming $\Sigma$-subspace. What we mean by this statement is in the content of the following theorem.\smallskip

\begin{theorem}\label{T:unique} Let $M$ be a JBW$^*$-triple. Then $M_\sigma$ is a norm-closed inner ideal in $M$. Moreover,
it is the only $1$-norming $\Sigma$-subspace which is also an inner ideal.
\end{theorem}
\noindent
{The theorem will be proved at the end of this section.\smallskip

The following technical result provides a characterization of $\sigma$-finite tripotents which is required later.
We recall that, given a tripotent $u$ in a JBW$^*$-triple $M$, there exists a complete tripotent $w\in M$ such
that $u\leq w$ (see \cite[Lemma 3.12$(1)$]{Horn87}).\smallskip

\begin{proposition}\label{p characterization of sigma finite tripotents} Let $u$ be a tripotent in a JBW$^*$-triple $M$.
The following statements are equivalent:
\begin{enumerate}[$(a)$]\item $u$ is $\sigma$-finite;
\item There exist a $\sigma$-finite tripotent $v$ and a complete tripotent $w$ in $M$ such that $v\leq w$ and $(w-v)\perp u$.
\end{enumerate}
\end{proposition}
}
\begin{proof} The implication $(a)\Rightarrow (b)$ is clear with $v=u$.\smallskip

$(b)\Rightarrow (a)$ Suppose there exist a $\sigma$-finite tripotent $v$ and a complete tripotent $w$ in $M$ such that $v\leq w$
and $(w-v)\perp u$.
Writing $w=v+(w-v)$ and using successively the orthogonality of $w-v$ to $u$ and to $v$ we obtain
$\{w,w,u\}=\{w,v,u\}=\{v,v,u\},$ and hence $L(w,w)u=L(v,v)u$, and similarly $\{w,u,w\}=\{v,u,v\}$. Since $w-v\perp M_2(v)\ni \{v,u,v\}$,  it follows that $P_2(w)(u) =Q(w)^2 (u) =\{w,\{v,u,v\},w\}= \{v,\{v,u,v\},v\}=P_2(v) (u)$.
Therefore,  $P_2(w) (u) = P_2 (v) (u)$ and $P_1(w) (u) =2L(w,w)(u)-2P_2(w)(u) = P_1 (v) (u)$.
\smallskip

The completeness of $w$ assures that $u = P_2(w) (u) + P_1 (w) (u) = P_2(v) (u) + P_1 (v) (u)$ lies in $M_2 (v)\oplus M_1(v)$.\smallskip

We shall show now that $u$ is $\sigma$-finite.  Arguing by contradiction, assume there is an uncountable family $(u_{j})_{j\in \Gamma}$ of mutually orthogonal non-zero tripotents in $M$ with $u_{j}\leq u$ for every $j$ (see \cite[\S 3]{EdRu98}). Since $u_j\in M_2 (u)$ for every $j$ and $u\perp (w-v)$, it follows that $u_j\perp (w-v)$ for every $j\in \Gamma$. Arguing as above we obtain $u_j \in M_2 (v)\oplus M_1(v)$, for every $j\in \Gamma$.\smallskip

Having in mind that $v$ is $\sigma$-finite, we can find a norm one functional  $\phi_{v}\in M_*$ whose support tripotent is $v$
(see \cite[Theorem 3.2]{EdRu98}).
By Lemma~\ref{l hilbertian norm}, $\phi_{v}$ gives rise to a norm $\|\cdot\|_{\phi_v}$ on $M_2(v)\oplus M_1(v)$ defined by
$\|x\|_{\phi_v}=(\phi_v\{x,x,v\})^{1/2}$ ($x\in M_2(v)\oplus M_1(v)$).
As $u_j$ is a non-zero element in $M_2(v)\oplus M_1(v)$
by the preceding paragraph, we obtain $$ \phi_v\{u_j,u_j,v\}=\|u_j\|^2>0\,.$$

Therefore, there exists a positive constant $\Theta$ and an uncountable subset $\Gamma'\subseteq \Gamma$ such that $\phi_{v} \{u_{j},u_{j},v\}>\Theta$ for all $j\in \Gamma'$. Thus, for each natural $m$ we can find $j_1\neq j_2\neq \ldots\neq j_m\in \Gamma'$. Since the elements $u_{j_1},\ldots, u_{j_m}$ are mutually orthogonal, we get $$1 =\left\| \sum_{k=1}^m u_{j_k} \right\|^2\geq \left\| \sum_{k=1}^m u_{j_k} \right\|^2_{\phi_{v}} = \phi_{v} \left\{\sum_{k=1}^m u_{j_k}, \sum_{k=1}^m u_{j_k}, v \right\} $$ $$= \sum_{k=1}^m \phi_{v} \left\{ u_{j_k}, u_{j_k}, v \right\} > m \Theta,$$ which is impossible.
\end{proof}

To prove that $M_\sigma$ is an inner ideal, we need another representation of $M$. To this end fix a complete tripotent $e\in M$.
Applying Theorem 3.4$(ii)$ in \cite{EdRu98} we can find a family $(e_{\lambda})_{\lambda\in\Lambda}$ of mutually orthogonal
$\sigma$-finite tripotents in $M$ satisfying $\displaystyle e= \sum_{\lambda\in \Lambda} e_{\lambda}$. For each $x\in M$
let us define $$\Lambda_{x}:=\{ \lambda\in \Lambda : L(e_{\lambda},e_{\lambda}) (x)\neq 0\}.$$

\begin{proposition}\label{p inner ideal of sigma finite range elements} In the conditions above,  $$M_{\sigma} =\{ x\in M : \Lambda_{x} \hbox{ is countable }\},$$
and $M_{\sigma}$ is a norm-closed inner ideal of $M$.
\end{proposition}

\begin{proof} Denote the set on the right-hand side by $M'_\sigma$. By the linearity of the Jordan product in the third variable it follows that $M'_\sigma$ is a linear subspace.
To show that it is an inner ideal, take $x,z\in M'_{\sigma}$ and $y\in M$. For each $\lambda\in \Lambda\backslash (\Lambda_{x}\cup\Lambda_z)$, we deduce via Jordan identity, that
$$\begin{aligned}L(e_{\lambda},e_{\lambda}) \J xyz &=\J{L(e_{\lambda},e_{\lambda}) x}yz - \J x{L(e_{\lambda},e_{\lambda})y}z +
\J xy{L(e_{\lambda},e_{\lambda})z}\\& =- \J x{L(e_{\lambda},e_{\lambda})y}z.\end{aligned}$$
Moreover, since $L(e_{\lambda},e_{\lambda}) x=L(e_{\lambda},e_{\lambda}) z=0$, we get $x,z\in M_{0} (e_{\lambda})$.
Since $P_0(e_{\lambda})y$ is in the $0$-eigenspace of $L(e_{\lambda},e_{\lambda})$ we have that
$L(e_{\lambda},e_{\lambda}) (y) \in M_1 (e_{\lambda})\oplus M_2 (e_{\lambda})$
and hence $\J x{L(e_{\lambda},e_{\lambda})(y)}z =0$ by Peirce arithmetic.
We have shown that $\Lambda_{\J xyz}\subseteq \Lambda_x\cup\Lambda_z,$ and thus $\Lambda_{\{x,y,z\}}$ is countable,
which proves that $\{x,y,z\}\in M'_{\sigma}$ and hence $M'_{\sigma}$ is an inner ideal of $M$.\smallskip

We continue by showing that $M_\sigma\subset M'_\sigma$. We shall first prove that $M'_\sigma$ contains all $\sigma$-finite tripotents in $M$.
Let $u$ be a $\sigma$-finite tripotent in $M$. We want to show that the set $\Lambda_u$ is countable.
We assume, on the contrary, that $\Lambda_u$ is uncountable. Let $\phi_{u}\in M_*$ be a norm one functional whose support tripotent
is $u$. For every $\lambda\in \Lambda_u$, we have that $e_\lambda\not\in M_0(u)$ because otherwise we would have
$L(e_{\lambda},e_{\lambda})(u)=0$.
Consequently, as in the proof of
Proposition~\ref{p characterization of sigma finite tripotents}, we deduce that $\phi_{u} \{e_{\lambda},e_{\lambda},u\}>0$. We can thus find a positive constant $\Theta$ and an uncountable subset $\Lambda_u'\subseteq \Lambda_u$ such that $\phi_{u} \{e_{\lambda},e_{\lambda},u\}>\Theta$ for all $\lambda\in \Lambda_u'.$ As before, for each natural $m$ we can find $\lambda_1\neq \lambda_2\neq \ldots\neq \lambda_m\in \Lambda_{u}'$. Then, applying the orthogonality of the elements $e_{\lambda_j}$ we get $$1 =\left\| \sum_{j=1}^m e_{\lambda_j} \right\|^2\geq \left\| \sum_{j=1}^m e_{\lambda_j} \right\|^2_{\phi_{u}} = \phi_{u} \left\{\sum_{j=1}^m e_{\lambda_j}, \sum_{j=1}^m e_{\lambda_j}, u \right\} $$ $$= \sum_{j=1}^m \phi_{u} \left\{ e_{\lambda_j}, e_{\lambda_j}, u \right\} > m \Theta,$$ which gives a contradiction. This proves that $\Lambda_u$ is countable, and hence $u\in M'_{\sigma}$.\smallskip

Let us now assume that $x$ is any element of $M_\sigma$. Then its range tripotent, $r(x)$, is $\sigma$-finite and hence  $r(x)\in M'_{\sigma}$ by the previous paragraph. Since $x\in M_2(r(x))$ is a positive and hence self-adjoint element, we have
$\J{r(x)}x{r(x)}=x$ and hence $x\in M'_\sigma$ as $M'_\sigma$ is an inner ideal.
This shows that $M_\sigma\subset M'_\sigma$.\smallskip

Conversely, let $x\in M'_{\sigma}$. In this case the set $\Lambda_{x}$ is countable. The tripotent $\displaystyle u=\hbox{w$^*$-}\sum_{\lambda\in \Lambda_x} e_{\lambda}$ is $\sigma$-finite in $M$, $e= u + v$, where $\displaystyle v= \hbox{w$^*$-}\sum_{\lambda\in \Lambda\backslash \Lambda_x} e_{\lambda}$ is another tripotent in $M$ with $u\perp v$. Since $\{e_{\lambda},e_{\lambda},x\}=0$ for all ${\lambda\in \Lambda\backslash \Lambda_x}$, it follows from the separate weak$^*$-continuity of the triple product of $M$ that $\{v,v,x\}=0$, that is, $x\in M_0 (v)$.
Hence also $r(x)\in M_0(v)$ (as $M_0(v)$ is a JBW$^*$-subtriple of $M$). It follows that $r(x)\perp v$ and hence $r(x)$ is $\sigma$-finite by Proposition \ref{p characterization of sigma finite tripotents}.\smallskip

We finally observe that, by Theorem \ref{t predual of JBW*-triples are 1-Plichko without submodels}, $M_{\sigma}$ is a $\Sigma$-subspace and hence it is norm-closed (cf. Lemma \ref{L:sigmasubspace}$(i)$). This completes the proof.
\end{proof}

We are now ready to prove the main theorem of this section.

\begin{proof}[Proof of Theorem~\ref{T:unique}]
$M_\sigma$ is a norm-closed inner ideal by Proposition~\ref{p inner ideal of sigma finite range elements}.
Let us prove the uniqueness.\smallskip

Let $I$ be {an  inner ideal} which is a $1$-norming  $\Sigma$-subspace. We will show that $I$ contains all sigma-finite
tripotents. Let $e\in M$ be a sigma-finite tripotent, $\phi\in M_*$ a normal functional of norm $1$ such that $e$ is the
support tripotent of $\phi$. By Lemma~\ref{L:sigmasubspace}(v) there is $x\in I$ of norm $1$ with $\phi(x)=1$.
Further, we get $r(x)\in I$. Indeed, $r(x)$ is contained in the weak$^*$-closure of the JB$^*$-subtriple of $M$ generated by $x$.
Since this subtriple is norm-separable, we get $r(x)\in I$ by Lemma~\ref{L:sigmasubspace}(i).\smallskip

In order to show $e\in I$ it is enough to show that $e\le r(x)$. By \eqref{eq order in support tripotents} it is enough to prove that $\phi(r(x))=1$.
Proposition 2.5 in \cite{Pe15} assures that $\phi(x^{[\frac{1}{2n+1}]}) =\phi(x)^{[\frac{1}{2n+1}]}=1,$ for all natural $n$. Since $\phi$ is a normal functional and $(x^{[\frac{1}{2n+1}]})\to r(x)$ in the weak$^*$ topology of $M$, it follows that $\phi (r(x)) =1$, as we desired.\smallskip

Now, if $z\in M_\sigma$ is arbitrary, then there is a $\sigma$-finite tripotent $f\in M$ with $z\in M_2(f)$.
By the above we have $f\in I$. Since $I$ is an inner ideal, we conclude that $M_2(f)\subset I$, and hence $z\in X$.\smallskip

Therefore, $M_\sigma\subset I$. Lemma~\ref{L:sigmasubspace}(iii) now shows that $M_\sigma=I$.
\end{proof}

\begin{remark}\label{remark proof with submodels}
It is possible to give a shorter proof of the fact that the predual of a JBW$^*$-triple is 1-Plichko by using the main result
of \cite{BoHamKal2016} at the cost of applying elementary submodels theory. However, this alternative argument does not yield $M_\sigma$ as a concrete description of a $\Sigma$-subspace.
We shall only sketch this variant:\smallskip

First, it is not too difficult to modify the decomposition \eqref{eq decomposition of JBW*-triples} by writing
\begin{equation}\label{eq decomposition of JBW*-triples 2} M = \left(\bigoplus_{j\in \mathcal{I}} A_j \overline{\otimes} G_j \right)_{\ell_{\infty}} \oplus_{\ell_{\infty}} N  \oplus_{\ell_{\infty}}  p V,
\end{equation}
where each $A_j$ is a commutative von Neumann algebra, each $G_j$ is a finite dimensional Cartan factor,
$p$ is a projection in a von Neumann algebra $V$, and $N$ is a JBW$^*$-algebra.\smallskip

Second, an almost word-by-word adaptation of the proof of \cite[Theorem 1.1]{BoHamKalIsrael} shows that the predual of $pV$ is 1-Plichko {\rm(}compare Proposition \ref{P:summandsbicontractive}{\rm)}.
So is the predual of $N$ by the main result of \cite{BoHamKal2016}. Finally, the summands $A_j \overline{\otimes} G_j$ are seen to have 1-Plichko predual
as in the proof of \ref{t Bochner 1-Plichko} {\rm(}or by an easier argument using the finite dimensionality of $C_j${\rm)}, and the stability of 1-Plichko spaces by $\ell_{1}$-sums {\rm(}\cite[Theorem 4.31$(iii)$]{Kal2000} or Lemma \ref{L:ellinftysum}{\rm)} allows us to conclude.
\end{remark}

\section{The case of real JBW$^*$-triples}\label{sec:6}

Introduced by J.M. Isidro, W. Kaup, and A. Rodr{\'i}guez (see \cite{IsKaRo95}), \emph{real JB$^*$-triples} are, by definition, the closed real subtriples of JB$^*$-triples. Every complex JB$^*$-triple is a real JB$^*$-triple when we consider the underlying real Banach structure. Real and complex C$^*$-algebras belong to the class of real JB$^*$-triples. An equivalent reformulation asserts that real JB$^*$-triples are in one-to-one correspondence with the real forms of JB$^*$-triples. More precisely, for each real JB$^*$-triple $E$ there exist a (complex) JB$^*$-triple $E_{c}$ and a period-2 conjugate-linear isometry (and hence a conjugate-linear triple isomorphism) $\tau: E_c\rightarrow E_c$ such that $E=\{b\in E_c\::\:\tau(b)=b\}$. The JB$^*$-triple $E_{c}$ identifies with the
complexification of $E$ (see \cite[Proposition 2.2]{IsKaRo95} or \cite[Proposition 4.2.54]{CabRodPal2014}). In particular, every JB-algebra (and hence the self-adjoint part, $A_{sa}$ of every C$^*$-algebra $A$) is a real JB$^*$-triple.\smallskip

Henceforth, for each complex Banach space $X$, the symbol $X_{\mathbb{R}}$ will denote the underlying real Banach space.\smallskip

In the conditions above we can consider another period-2 conjugate-linear isometry $\tau^{\sharp} : E_{c}^{*} \rightarrow
E_{c}^{*}$ defined by $$\tau^{\sharp}(\varphi) (z) := \overline{\varphi (\tau (z))}\ \ (\varphi\in E_{c}^{*}).$$ It is further known that
the operator $$(E_c^{*})^{\tau^{\sharp}}
\rightarrow (E_c^{\tau})^{*}, \ \ \varphi \mapsto \varphi|_{E}$$ is an isometric real-linear bijection, where $(E_c^{*})^{\tau^{\sharp}} := \{ \varphi \in E_{c}^{*} : \tau^{\sharp} (\varphi)=\varphi \}$.\smallskip

A real JBW$^*$-triple is a real JB$^*$-triple which is also a dual Banach space (\cite[Definition 4.1]{IsKaRo95} and \cite[Theorem 2.11]{MarPe}). It is known that every real JBW$^*$-triple admits a unique (isometric) predual and its triple product is separately weak$^*$-continuous (see \cite[Proposition 2.3 and Theorem 2.11]{MarPe}). Actually, by the just quoted results, given a real JBW$^*$-triple $N$ there exists a JBW$^*$-triple $M$ and a weak$^*$-to-weak$^*$ continuous period-2 conjugate-linear isometry $\tau: M\to M$ such that $N = M^{\tau}$. The mapping $\tau^{\sharp}$ maps $M_*$ into itself, and hence we can identify $(M_*)^{\tau^{\sharp}}$ with $N_{*} =(M^{\tau})_{*}$. We can also consider a weak$^*$-continuous real-linear bicontractive projection $P=\frac12 (Id+\tau)$ of $M$ onto $N= M^{\tau}$, and a bicontractive real-linear projection of $M_*$ onto $N_*$ defined by $Q=\frac12 (Id+\tau^{\sharp})$. From now on, $N$, $M$, $\tau$, $P,$ and $Q$ will have the meaning explained in this paragraph.\smallskip

Due to the general lack for real JBW$^*$-triples of the kind of structure results established by Horn and Neher for JBW$^*$-triples in \cite{Ho2,HoNe}, the proofs given in section \ref{sec:4} cannot be applied for real JBW$^*$-triples. Despite of the limitations appearing in the real setting, we shall see how the tools in previous section can be applied to prove that preduals of real JBW$^*$-triples are 1-Plichko spaces too.\smallskip

We shall need to extend the concept of $\sigma$-finite tripotents to the setting of real JBW$^*$-triples. The notions of tripotents, Peirce projections, Peirce decomposition are perfectly transferred to the real setting. The relations of orthogonality and order also make sense in the set of tripotents in $N$ (cf. \cite{IsKaRo95,MarPe}). Furthermore, for each tripotent $e$ in $N$, $Q(e)$ induces a decomposition of $N$ into $\RR$-linear subspaces satisfying $$N = N^1(e) \oplus
N^0(e) \oplus N^{-1}(e),$$ {where} $N^k (e) := \{
x\in N :\ Q(e)x = kx\},$  $$N_2(e) = N^1(e) \oplus N^{-1}(e)
 \ ,\quad
N^0(e) = N_1(e)\oplus N_0(e),$$ $$\{N^j(e),N^k(e),N^\ell(e)\}
\subset N^{jk\ell}(e) \hbox{ if } jk\ell\ne 0 ,\ j,k,\ell \in\{ 0,\pm 1\}, \hbox{ and zero otherwise}.$$ The natural projection of $N$ onto $N^{k}
(e)$ is denoted by $P^{k}(e)$. It is also known that $P^{1} (e)$, $P^{-1} (e)$, and $P^{0} (e)$ are all weak$^*$-continuous. The subspace $N^{1} (e)$ is a weak$^*$-closed Jordan subalgebra of the JBW-algebra $(M_2(e))_{sa}$, and hence $N^{1} (e)$ is a JBW-algebra.\smallskip

Given a normal functional $\phi\in N_*$, there exists a normal functional $\varphi\in M_*$ satisfying $\tau^{\sharp} (\varphi) =\varphi$ and $\varphi|_{N} =\phi$. Let $e(\varphi)$ be the support tripotent of $\varphi$ in $M$. Since $1=\varphi(e(\varphi)) = \overline{\varphi(\tau(e(\varphi)))} =\varphi(\tau(e(\varphi)))$, we deduce that $\tau(e(\varphi)) \geq e(\varphi)$. Applying that $\tau$ is a triple homomorphism, we get $e(\varphi) = \tau^2 (e(\varphi)) \geq \tau (e(\varphi)) \geq e(\varphi)$, which proves that $e(\varphi) =\tau(e(\varphi))\in N.$ That is, the support tripotent of a $\tau^{\sharp}$-symmetric normal functional $\varphi$ in $M_*$ is $\tau$-symmetric. The tripotent $e(\varphi)$ is called the support tripotent of $\phi$ in $N$, and it is denoted by $e(\phi)$. It is known that $\phi =\phi P^{1}(e(\phi))$ and $\phi|_{N^{1} (e(\phi))}$ is a faithful positive normal functional on the JBW-algebra $N^{1} (e(\phi))$ (compare \cite[Lemma 2.7]{PeSta}).\smallskip

As in the complex setting, a tripotent $e$ in $N$ is called \emph{$\sigma$-finite} if $e$ does not majorize an uncountable orthogonal subset of tripotents in $N$. The real JBW$^*$-triple $N$ is called \emph{$\sigma$-finite} if every tripotent in $N$ is {$\sigma$-finite}.

\begin{proposition}\label{p charact sigma-finite tripotents in real} In the setting fixed for this section, let $e$ be a tripotent in $N$. The following are equivalent:
\begin{enumerate}[$(a)$] \item $e$ is $\sigma$-finite in $N$;
\item $e$ is $\sigma$-finite in $M$;
\item $e$ is the support tripotent of a normal functional $\phi$ in $N_*$;
\item $e$ is the support tripotent of a $\tau^{\sharp}$-symmetric normal functional $\varphi$ in $M_*$.
\end{enumerate}
Consequently, for $$N_{\sigma} := \{x\in N : \hbox{there exists a $\sigma$-finite tripotent $e$ in } N \hbox{ with } \{e,e,x\} =x\}$$ we have $$N_{\sigma}=\{x\in M_{\sigma} : \tau (x) =x\} =N\cap M_{\sigma},$$ and the following are equivalent: \begin{enumerate}[$(i)$] \item $M$ is $\sigma$-finite {\rm(}i.e. $M_{\sigma}=M${\rm)}; \item $N$ is $\sigma$-finite {\rm(}i.e. $N_{\sigma}=N${\rm)}; \item $N$ contains a complete $\sigma$-finite tripotent.
\end{enumerate}
\end{proposition}

\begin{proof} The implication $(b)\Rightarrow (a)$ and the equivalence $(c)\Leftrightarrow (d)$ are clear. The implication $(d)\Rightarrow (b)$ follows from \cite[Theorem 3.2]{EdRu98}. To see $(a)\Rightarrow (d)$, let us assume that $e$ is $\sigma$-finite in $N$. Clearly $e$ is the unit in the JBW-algebra $N^{1} (e)$, and since every family of mutually orthogonal projections in this algebra is a family of mutually orthogonal tripotents in $N$ majorized by $e$, we deduce that $e$ is a $\sigma$-finite projection in $N^{1}(e).$ Theorem 4.6 in \cite{EdRu85} assures the existence of a faithful normal state $\phi$ in $(N^{1}(e))_*$. By a slight abuse of notation, the symbol $\phi$ will also denote the functional $\phi P^{1} (e)$. Clearly $\phi\in N_*$ and $\phi|_{N^{1} (e)}$ is a faithful normal state.\smallskip

By the arguments above, there exists a $\tau^{\sharp}$-symmetric normal functional $\varphi$ in $M_*$ such that $\varphi|_{N} =\phi$. Let $e(\varphi)$ be the support tripotent of $\varphi$ in $M$. We have also commented before this proposition that $\tau(e(\varphi))=e(\varphi)$ (i.e. $e(\varphi)\in N$) because $\phi$ is $\tau^{\sharp}$-symmetric. Since $\varphi (e) =\phi (e) =1$, we deduce that $e\geq e(\varphi)$. Therefore $e(\varphi)$ is a projection in the JBW-algebra $N^{1} (e)$. Furthermore, $\phi (e(\varphi)) = \varphi (e(\varphi))=1$ and the faithfulness of $\phi|_{N^{1} (e)}$ show that $e=e(\varphi)$. This proves the equivalence of $(a)$, $(b)$, $(c)$ and $(d)$. The equality $N_{\sigma} = N\cap M_{\sigma}$ is clear from the first statement.\smallskip

Since a complete tripotent in $N$ is a complete tripotent in $M$, the rest of the statement follows from the previous equivalences and \cite[Theorem 4.4]{EdRu98}.
\end{proof}

We can prove now our main result for preduals of real JBW$^*$-triples.

\begin{theorem}\label{t predual real JBW*-triples} The predual of any real JBW$^*$-triple $N$ is a $1$-Plichko space. 
Moreover, $N_*$ is is weakly Lindel\"{o}f determined if and only if $N$ is $\sigma$-finite. In the latter case $N_*$ is even weakly compactly generated.
\end{theorem}

\begin{proof} We keep the notation fixed for this section with $N$, $M$ and $\tau$ as above. There exists a canonical isometric identification of $M_{\mathbb{R}}$ with $((M_*)_{\mathbb{R}})^*$, where any $x\in M_{\mathbb{R}}$ acts on $(M_*)_{\mathbb{R}}$ by the assignment $\omega\mapsto\operatorname{Re}\omega(x)$ ($\omega\in (M_*)_{\mathbb{R}}$). Thus $(M_*)_{\mathbb{R}}$ is a real $1$-Plichko space and $M_{\sigma}$ is again a $1$-norming $\sigma$-subspace by Theorem \ref{t predual of JBW*-triples are 1-Plichko without submodels} and \cite[Proposition 3.4]{Kal2005}.\smallskip

In view of Lemma~\ref{L:quotient} to prove that the predual of $N$ is $1$-Plichko, it is enough to show that ${\mathcal B}_{N}\cap M_\sigma$ is weak$^*$-dense in ${\mathcal B}_N$. Since $M_{\sigma}$ is a 1-norming subspace we can easily see that $\ball_{M_{\sigma}}$ is weak$^*$-dense in $\ball_{M}$. Take an element $a\in {\mathcal B}_N\subset {\mathcal B}_M$. Then there exists a net $(a_{\lambda}) \subset {\ball_{M_{\sigma}}}$ converging to $a$ in the weak$^*$-topology of $M$. Since $\tau$ is weak$^*$-continuous and $M_{\sigma}$ is a norm-closed $\tau$-invariant subspace of $M$, we can easily see that $(\frac{a_{\lambda}+\tau(a_{\lambda})}{2})\to a$ in the weak$^*$-topology of $M$, where $(\frac{a_{\lambda}+\tau(a_{\lambda})}{2})\subset {\ball_{N_{\sigma}}} ={\mathcal B}_{N}\cap M_\sigma$, which proves the desired weak$^*$-density.\smallskip

For the last statement, we observe that $N$ is $\sigma$-finite if and only if $M$ is (see Proposition \ref{p charact sigma-finite tripotents in real}), and hence the desired equivalence follows from Theorem \ref{t predual of JBW*-triples are 1-Plichko without submodels} and the results presented in sections \ref{sec:4} and \ref{sec:6}. We also note that $N$ $\sigma$-finite implies $M$ $\sigma$-finite implies $M_*$ WCG implies
$N_*$ WCG, being a complemented subspace.
\end{proof}

We can rediscover the following two results in \cite{BoHamKalIsrael} and \cite{BoHamKal2016} as corollaries of our last theorem.

\begin{corollary}\label{c Msa}\cite[Theorem 1.4]{BoHamKalIsrael} Let $W$ be a von Neumann algebra. Then the predual, $(W_{sa})_*$, of the self-adjoint part, $W_{sa},$ of $W$ is a 1-Plichko space. Moreover, $(W_{sa})_*$ is weakly
Lindel\"{o}f determined if and only if $W$ is $\sigma$-finite. In the latter case $W_*$ and $(W_{sa})_*$ are even weakly compactly generated. $\hfill\Box$
\end{corollary}

\begin{corollary}\label{c JBW preduals}\cite[Theorem 1.1]{BoHamKal2016}  The predual of any JBW-algebra $J$ is $1$-Plichko. Moreover, $J_*$ is is weakly Lindel\"{o}f determined if and only if $J$ is $\sigma$-finite. In the latter case $J_*$ is even weakly compactly generated.$\hfill\Box$
\end{corollary}\smallskip

\noindent{\sc Acknowledgement} The last mentioned author thanks A. Defant for helpful remarks.

\end{document}